\documentclass{amsart}

\usepackage{amssymb}
\usepackage{amsmath}

\usepackage{macros}
\standardsettings
\colorcommentstrue
\usetikzlibrary{positioning}
\usepackage{games}

\usepackage[top=1in, left=1in, right=1in, bottom=1in]{geometry}
\usepackage{epsfig}
\DeclareMathAlphabet{\mathpzc}{OT1}{pzc}{m}{it}

\draftfalse

\renewcommand{\I}{\mathrm{\Romannum{1}}}
\renewcommand{\II}{\mathrm{\Romannum{2}}}
\newcommand{\setof}[1]{\left\{#1\right\}}
\newcommand{\suchthat}{\colon~}
\newcommand{\abs}[1]{\left | #1 \right |}
\newcommand*\concat[0]{{}^\smallfrown}

\newcommand{\sM}{\mathcal{M}}
\newcommand{\res}{\restriction}
\newcommand{\bG}{\boldsymbol{\Gamma}}
\newcommand{\bS}{\boldsymbol{\Sigma}}
\newcommand{\bP}{\boldsymbol{\Pi}}
\newcommand{\bD}{\boldsymbol{\Delta}}

\newcommand{\zfc}{\mathsf{ZFC}}

\newcommand{\ad}{\mathsf{AD}}

\newcommand{\dom}{\text{dom}}

\begin{document}
\title{The Measure Game}
\authorlogan
\authorlior
\authorsteve
\authorhouston
\authordavid

\
\begin{abstract}
We study a game first introduced by Martin \cite{Martin1998}
(actually we use a slight variation of this game)
which plays a role for measure analogous to the Banach-Mazur game for category.
We first present proofs for  the basic connections between this game and measure, and then
use the game to prove fundamental measure theoretic results such as
Fubini's theorem, the Borel-Cantelli lemma, and a general unfolding result for the game
which gives, for example, the measurability of $\bS^1_1$ sets.
We also use the game
to give a new, more constructive, proof of a strong form of the R\'{e}nyi-Lamperti lemma,
an important result in probability theory
with many applications to number theory. The proofs we give are all direct combinatorial
arguments using the game, and do not depend on known measure theoretic arguments.
\end{abstract}

\maketitle

\section{Introduction}
Measure and category are two of the fundamental notions of analysis, topology,
and descriptive set theory. In the case of category there is a {\textsl{natural}}
game (this notion is defined below), the Banach-Mazur game (see \cite{Oxtoby}), associated to the
topological category of a set. Many important properties of 
category (e.g.\ the Kuratowski-Ulam theorem, unfolding arguments)
have direct proofs in terms of this game. 

In this paper we study a natural game corresponding to measure introduced
by Martin (c.f.\ Theorems 8 and 9 of
\cite{Martin1998};
see also the discussion in Rosendal \cite{rosendal}).
We use this game to directly prove the main results about measure analogous
to the above mentioned facts for category, and obtain new results such as a
different, more constructive, proof of a strong form of the
R\'{e}nyi-Lamperti theorem from probability theory. 
Our results here are given in the case of Borel probability measures on Polish spaces,
and have the advantage of being ``constructive'' in the sense that we are able to
explicitly exhibit strategies in the game, which in turn can be used to construct points in the measure space.
This of course cannot be done in complete generality without some strong hypotheses on the space.

In \S\ref{sec:basic} we establish the basic properties of the game
and relate them to measurability. In \S\ref{sec:bc} we use the game to
prove the classical Borel-Cantelli theorem.
We then use the game to
prove a significantly stronger result, a strong form of the R\'{e}nyi-Lamperti theorem
(due to \cite{kochen-stone} and \cite{spitzer}, see also \cite{petrov}).
The latter result has found numerous applications in
number theory (see for example \cite{harmon}).
Here the assumption of mutual independence
of the events is replaced by a weaker requirement. In \S\ref{sec:fubini} we use the game to give a direct proof
of Fubini's theorem (the measure analog of the Kuratowski-Ulam theorem). 
In \S\ref{sec:unfolding} we use the game to
give a combinatorial proof of an unfolding result for the measure game.
This has several consequences such as showing directly that $\bS^1_1$ sets
are universally measurable, showing $\bP^1_n$-determinacy gives the
measurability of $\bS^1_{n+1}$ sets, and proving  continuous uniformization
results for measure. 

We note that \S \ref{sec:bc}, \S\ref{sec:fubini}, and \S\ref{sec:unfolding} are all independent and those more interested in the measure-theoretic applications of the game can focus on \S\ref{sec:bc} and \S\ref{sec:fubini}, while those interested in descriptive set theoretic applications could skip directly to \S \ref{sec:unfolding}.  The material in \S\ref{sec:basic} establishes the basic properties of the game for the sake of completeness, and while the results therein are important and are used throughout the paper, their proofs are relatively straightforward.

Now we establish notation.  Let $\omega=\{0,1,\dots\}$ denote the set of natural numbers. Let $2^\omega$ denote the space of infinite binary sequences with the product of discrete topologies on $2=\{ 0,1\}$.  We note that $2^\omega$ is homeomorphic to the Cantor middle-thirds sets in $\R$. We also let $\omega^\omega$ denote the space of infinite sequences of natural numbers, (the {\em Baire space}), which is homeomorphic to the space of irrationals. A subset $T \subseteq X^{<\omega}$ is a \emph{tree} if it is closed under taking initial segments.  We let $[T]=\{u \in X^\omega \colon~\text{for all}~x \subset u, x \in T\}$ denote the set of branches through $T$.

By a {\em game} on a set $X$ with {\em payoff set} $A\subseteq X^\omega$ we mean a two-player game where 
$\I$ makes moves $x_{2i}\in X$ and $\II$ makes moves $x_{2i+1}\in X$ as shown:

\begin{figure}[ht]
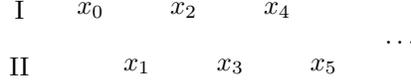

\begin{games}[player spacing=0.25cm]
\game[dots]{ {$x_0$} {$x_1$} {$x_2$} {$x_3$} {$x_4$} {$x_5$}}
\end{games}
\caption{A general game on a set $X$} \label{fig:gg}
\end{figure}
\noindent $\I$ wins the run of the game iff $x=(x_0,x_1,\dots) \in
A$. When convenient, we will also define the game by the winning
condition for player $\II$. Often the exact payoff set is defined
implicitly by giving a list of ``rules'' the players must follow (the
first player to violate one of the rules loses) and then defining the
payoff set for runs of the game that follow the rules.

The axiom $\ad_X$ is the assertion
that all games on $X$, for any payoff set $A\subseteq X^\omega$, are
determined, that is, one of the players has a winning strategy
(defined in the obvious way).  $\ad$, the axiom of determinacy, is the
axiom $\ad_\omega$. Our main game is equivalent to an integer
game, and so is determined from $\ad$. 
$\ad$ contradicts the axiom of  choice but
does not contradict weaker forms of choice such as countable and
dependent choice, and is a reasonable axiom within the realm of
definable sets.

By a {\em pointclass} $\bG$ we
mean a collection of subsets of Polish spaces which is closed under
inverse images by continuous functions $f \colon X\to Y$ (i.e., if
$A\subseteq Y$ is in $\bG$, then $B=f^{-1}(A)$ is in $\bG$).  Examples
include all the levels $\bS^0_\alpha$, $\bP^0_\alpha$, $\bD^0_\alpha$
of the Borel hierarchy as well as $\bS^1_1$ (the analytic sets),
$\bP^1_1$ (the co-analytic sets), and most naturally occurring
collections of definable sets.  A fundamental result of Martin
\cite{Martin_determinacy} says that all Borel games (on any set $X$)
are determined (Borel refers to the product of the discrete topology
on $X$).  The determinacy of $\bS^1_1$ games, however, is not a
theorem of $\zfc$. We let $\det(\bG)$ denote the statement that
all games on $\omega$ with payoff set $A \in \bG$ are determined.

In the case of category, for any pointclass $\bG$ the determinacy of the Banach-Mazur games for sets
in the pointclass $\bG$ is equivalent to the Baire property of sets in $\bG$.
The Banach-Mazur game also permits an ``unfolding,'' an important property which has
several consequences. In particular, the Banach-Mazur game for $\bS^1_1$ sets unfolds
to a closed game, which is therefore determined. This shows in $\zfc$ that every
$\bS^1_1$ set has the Baire property. A related fact is the Kuratowski-Ulam theorem
for category, which is the analog of the Fubini theorem for measure. This theorem can
also be proved using the Banach-Mazur game. Thus we have a natural game associated
to the notion of category which can be used to established the main results
about this notion. We now introduce a game, the measure game, which we show
plays an analogous role for measure.

If $X$ is an uncountable  Polish space then it is well-known that there is a
Borel isomorphism $\pi$ between 
$X$ and $2^\omega$. If $\nu$ is a Borel probability measure on $X$, then $\mu=\pi(\nu)$ 
is a Borel probability measure on $2^\omega$. Thus, for all arguments it suffices to consider 
the case $X=2^\omega$, and $\mu$ a Borel probability measure on $2^\omega$. It is also possible 
to define versions of the game directly on arbitrary Polish spaces.

So let $\mu$ be a fixed Borel probability measure on $2^\omega$ and let $\mu^*$, $\mu_*$ denote the corresponding outer and inner measure respectively, i.e.
\[\mu^*(A)=\inf\setof{\mu(U) \suchthat U~\text{is open},~ A \subseteq U}\]
\[\mu_*(A)=\sup\setof{\mu(F) \suchthat F~\text{is closed},~ F \subseteq A}.\]
Let $A\subseteq 2^\omega$ and $s\in [0,1)$.
For a finite sequence $t \in 2^{<\omega}$, let $N_t$ denote the set $\setof{x \in 2^\omega \suchthat t \subseteq x}$.
Consider the game $G(s,A)$ played as follows.  
In each round, player $\I$ offers two numbers $m_{t\concat 0}$ and $m_{t \concat 1}$
assigned to $N_{t \concat 0}$ and $N_{t \concat 1}$ respectively.
For notational convenience, we also set $m_\emptyset=m_0+m_1$.
Player $\II$ then selects
either $t\concat 0$ or $t\concat 1$ and play continues on to the next round.
A typical run of the game is exhibited in
Figure~\ref{fig:mg}.
\begin{figure}[ht]
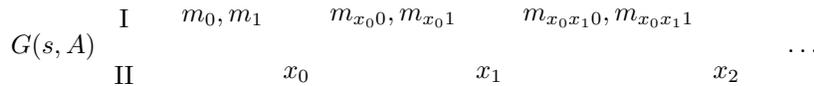

\begin{games}
\game[player spacing=0.25cm, game label={$G(s, A)$}, dots]{
    {$m_0, m_1$}{$x_0$}
    {$m_{x_0 0}, m_{x_0 1}$}{$x_1$}
    {$m_{x_0 x_1 0}, m_{x_0 x_1 1}$}{$x_2$}
    }
\end{games}
\caption{The measure game} \label{fig:mg}
\end{figure}

We require that player $\I$ must follow the rules:
\begin{itemize}
\item $\forall t \in 2^{<\omega},~ m_t \in \mathbb{Q}$ 
\item $m_0+m_1 > s$
\item $\forall t \in 2^{<\omega},~ 0\leq m_t \leq  \mu(N_t)$
\item $\forall t \in 2^{<\omega}, ~0 < \mu(N_t) \Rightarrow m_{t} < \mu(N_t)$
\item $\forall t \in 2^{<\omega}, ~ m_{t\concat 0}+m_{t \concat 1}= m_t$
\end{itemize}
and player $\II$ must follow the rules:
\begin{itemize}
\item $\forall n \in \omega,~  x_n \in \setof{0, 1}$
\item $\forall n \in \omega,~  m_{x_0 x_1 \dots x_n} \neq 0$.
\end{itemize}
We declare that $\II$ wins a run of $G(s, A)$ if and only if $x = (x_0,x_1,\dots) \in A$, otherwise $\I$ wins.
The most significant rule is $\II$'s restriction that $m_{x_0 x_1 \dots x_n} \neq 0$, as this is how
player $\I$ exerts control over the outcome.

Note that as $\I$'s moves $m_t$ are from $\Q$, the game $G(s,A)$ is essentially an integer game, and
thus is determined from $\ad$. In fact, in Corollary~\ref{cor:gameeq} we point out that the
game $G(s,A)$ is equivalent to the version where $\I$'s moves $m_t$ are allowed to
be arbitrary reals. Thus $\ad_\R$  does not play a role in our discussions.

\section{Basic properties of the game} \label{sec:basic}

In this section we prove that this game characterizes $\mu$-measure in
the same way that the Banach-Mazur game characterizes category.
We first prove two lemmas which analyze winning strategies for players $\I$ and $\II$.
We then prove a theorem which gives the correspondence between the game and measure.
Finally, though we don't need it for the other results of the paper,
we show that the rational version of the game defined above is equivalent to the ``real version'' in which
$\I$'s moves are allowed to be real numbers.

We make use of the following simple concept. 

\begin{definition}
A function $M\colon 2^{<\omega} \to \mathbb{R}$ is a \emph{scaled measure} if $M(\emptyset) > 0$ and for
every  $t \in 2^{<\omega}$, 
\begin{itemize}
\item $\displaystyle 0 \leq M(t) \leq \mu(N_t)$
\item $M(t) = M(t \concat 0) + M(t \concat 1)$
\end{itemize}
If $T \subseteq 2^{<\omega}$ is a nonempty tree, then a scaled measure $M$ is a
\emph{scaled measure on $T$} if $t \not \in T \Rightarrow M(t)=0$. Also
by the {\emph support} of $M$ we mean the tree $T$ of $t$ such that
$M(t)>0$. 
\end{definition}

\begin{lemma} \label{lem:ws1}
Suppose $\sigma$ is  a winning strategy for $\I$ in the game $G(s,A)$. Then
$\sigma$ induces a scaled measure $M$ as follows. Let $M(\emptyset)=m_\emptyset$.
For $t \in 2^{<\omega}$,
let $\sigma(t)= (m_{t\concat 0}, m_{t \concat 1})$, provided $t$ follows the rules
against $\sigma$. In this case, let $M(t\concat i)= m_{t\concat i}$. Otherwise,
set $M(t\concat i)=0$. If $T$ is the support of $M$, then $[T] \subseteq A^c$
and $\mu([T])\geq M(\emptyset)>s$. 
\end{lemma}

\begin{proof}
It is clear by the rules of the game that $M$ as defined in the statement is a scaled
measure with $M(\emptyset)=m_0+m_1>s$. Since $\sigma$ is winning for $\I$, $[T] \subseteq A^c$,
where $T$ is the support of $M$. 
To see that $\mu([T]) \geq M(\emptyset)>s$, suppose that $\mu([T])<M(\emptyset)$.
For $n$ large enough we have that $\sum_{t \in T\cap 2^n} \mu(N_t) <M(\emptyset)$.
Due to the additivity of $M$ $M(\emptyset)= \sum_{t \in T\cap 2^n} M(t)$.  However by the rules of the game, $M(t)\leq
\mu(N_t)$ for all $t$. Thus, $M(\emptyset) \leq \sum_{t \in T\cap 2^n} \mu(N_t)< M(\emptyset)$,
a contradiction.
\end{proof}

We now consider the case where $\tau$ is a winning strategy for $\II$ in $G(s,A)$.
The following lemma (implicitly in Martin \cite{Martin1998}, see also Rosendal \cite{rosendal}),
introduces a useful notation which we will also use later.

\begin{lemma} \label{lemma:delta}
Let $\tau$ be a winning strategy for $\II$ in $G(s, A)$. Let $x_0 \dots x_n = t \in 2^{<\omega}$
and let $\hat{m} = \setof{m_{x_0\dots x_i\concat 0}, m_{x_0 \dots x_i\concat 1}}_{i < n}$
be any finite sequence of moves by $\I$ so that $\tau$ plays $x_0 \dots x_n$ in response to $\hat{m}$,
i.e., $x_i = \tau(\setof{m_{x_0 \dots x_j \concat 0}, m_{x_0 \dots x_j \concat 1}}_{j < i})$ for each $i$.

If we define 
\[\delta_\tau^{\hat{m}}(i) = \inf\left(\setof{m_{t \concat i} \suchthat \exists m_{t\concat (1-i)},~m_{t \concat 0} + m_{t \concat 1}=m_t~\text{and}~ \tau\left(\hat{m} \concat (m_{t \concat 0}, m_{t \concat 1})\right) = i} \cup \setof{\mu(N_{t \concat i})}\right)\]

Then 
\begin{itemize}
\item $\delta_\tau^{\hat{m}}(0)+\delta_\tau^{\hat{m}}(1) \leq r = \begin{cases}m_t & ~\text{if}~\abs{t}>0\\ s &~ \text{if}~ \abs{t}=0
\end{cases}$
\item $\displaystyle\delta_\tau^{\hat{m}}(i) \leq \mu(N_{t \concat i})$
\end{itemize} 
\end{lemma}
\begin{proof}
The second item is clear.  For the first item, suppose not. Then $\delta_\tau^{\hat{m}}(0)+\delta_\tau^{\hat{m}}(1) > r$.  Let $m'_i < \delta_\tau^{\hat{m}}(i)$ be rational so that $m'_0+m'_1 > r$.  In the case that $\abs{t}>0$, let $m_{t \concat i} = m'_i\frac{r}{m'_0 + m'_1}$, each of which will still be rational, as $r=m_t$ was some legal move by $\I$.  These will also be such that $m_{t \concat 0}+m_{t \concat 1} = r = m_t$.  If $\abs{t}=0$, then let $m_{t \concat i}=m'_i$.  In either case, $(m_{t \concat 0}, m_{t \concat 1})$ is a legal move at this position, but $m_{t \concat i} < \delta_\tau^{\hat{m}}(i)$ for both sides, which is a contradiction as $\tau$
must choose one of the two sides.
\end{proof}

The next lemma is the analog of Lemma~\ref{lem:ws1} for player $\II$.

\begin{lemma} \label{lemma:rosendal}
Suppose $\II$ wins $G(s, A)$ with winning strategy $\tau$.
Then for any $\epsilon>0$ there exists a tree $T_\epsilon \subseteq 2^{<\omega}$ such that
$\mu([T_\epsilon])> 1-s-\epsilon$ and 
such that every branch through $T_\epsilon$ is a run consistent with $\tau$, and so 
$\left[T_\epsilon\right]\subseteq A$. In particular,  $\mu_*(A)\geq 1-s$.
\end{lemma}

\begin{proof}
Let $\tau$ be a winning strategy for $\II$ in $G(s, A)$ and let $\epsilon>0$.
We build $T_\epsilon$ by induction on the levels of the tree. For each $u \in 2^{<\omega}$
we define a number $m_u\in [0, \mu(N_u)]$. $T_\epsilon$ will be the set of $u$
for which $m_u< \mu(N_u)$. For each $u \in T_\epsilon$ we will define an
associated play $p=p(u)=((m'_0,m'_1),u(0), (m'_{u(0),0}, m'_{u(0),1}), u(1),\dots, u(n-1))$
of the game $G(s,A)$ consistent with $\tau$ for which $\II$'s moves are equal to
$u$ and for every initial segment $t$ of $u$ we have $m'_t=m_t$. We also
maintain that for each level $n$ we have
\[
\sum_{u \in 2^n} m_u \leq s+ \frac{2^n-1}{2^n} \epsilon.
\]

To begin, let $m_\emptyset=s <\mu(2^\omega))=1$, and so $\emptyset \in T_\epsilon$.  
Consider now $\delta_\tau^\emptyset$ in the notation of Lemma~\ref{lemma:delta}, if $\delta_\tau^\emptyset(i)=\mu(N_i)$, then set $m_i=\mu(N_i)$.  
If $\delta_\tau^\emptyset(i)<\mu(N_i)$
then choose $m_i<\mu(N_i)$ such that
$0<m_i-\delta_\tau^\emptyset(i)<\frac{\epsilon}{4}$ and so that $\tau$
will choose $i$ in response to a move that assigns $m_i$ to the side
$N_i$. We clearly have that $m_0+m_1 \leq s+\frac{1}{2}\epsilon$.

In general, suppose the $m_u$ have been defined for $u \in 2^n$
satisfying the above properties. For any $u \in 2^n$ so that $m_u =\mu(N_u)$, set $m_{u \concat i}=\mu(N_{u\concat i})$ for $i \in 2$.  
Now suppose $u \in 2^n$ and
$m_u <\mu(N_u)$. By assumption we have a play $p=p(u)$ consistent with $\tau$
in which $\tau$'s last move $u(n-1)$ selected a side with
assigned value $m_u$. We apply Lemma~\ref{lemma:delta} to
$\tau$, $p(u)$, and $m_u$ to get $\delta_\tau^p=(\delta_\tau^p(0),\delta_\tau^p(1))$.
For $i$ such that $\delta_\tau^p(i)=\mu(N_{u\concat i})$, set $m_{u\concat i}=
\mu(N_{u\concat i})$. Otherwise we choose $m_{u\concat i} < \mu(N_{u\concat i})$
such that $0 < m_{u\concat i} -\delta_\tau^p(i) < \frac{\epsilon}{4^{n+1}}$. 
Since there only $2^{n+1}$ nodes at level $n+1$, we have that
$\sum_{u \in 2^{n+1}} m_u \leq (\sum_{u \in 2^n} m_u) + \frac{\epsilon}{2^{n+1}}$. 
This maintains the desired inequality.

This completes the definition of the tree $T_\epsilon$. For each node
$u \in T_\epsilon$ we have a corresponding play $p(u)$ following $\tau$.
By construction, if $u \subseteq v$ then $p(v)$ extends the play $p(u)$.
Thus, for any $x \in [T_\epsilon]$ there is a run of the game $G(s,A)$
following $\tau$ in which $\II$ has played $x$, and so $x \in A$.
For each level $n$ we have maintained that $\sum_{u \in 2^n} m_u < s+\epsilon$. 
Also, for $u \notin T_\epsilon$ we have $m_u=\mu(N_u)$, thus,
$\sum_{u \in 2^n\setminus T_\epsilon} \mu(N_u) < s+\epsilon$. It follows
that $\mu([T_\epsilon]) \geq 1- (s+\epsilon)$.

\end{proof}

\begin{theorem} [Martin\cite{Martin1998}] \label{thm:equiv}
Let $A \subseteq{2^\omega}$ and $s \in [0, 1)$.
\begin{equation*}
\begin{split}
\mu_* (A^c) > s ~ & \Leftrightarrow ~\I~\text{wins}~ G(s, A) \\
\mu^* (A^c) \leq s ~ & \Leftrightarrow ~\II~\text{wins}~ G(s, A)
\end{split}
\end{equation*}
\end{theorem}

\begin{proof}
First suppose $\I$ wins $G(s,A)$. From Lemma~\ref{lem:ws1}, a winning strategy $\sigma$ for $\I$
gives  a scaled measure $M$ on $2^\omega$ with $M(\emptyset)>s$. If $T$ is the support of
$M$, the Lemma also shows that $[T] \subseteq A^c$ and $\mu([T]) \geq M(\emptyset)>s$.
Thus $\mu_*(A^c) >s$. 

Suppose next that $\II$ has a winning strategy $\tau$ for $G(s,A)$. 
From Lemma~\ref{lemma:rosendal} we have that $\mu_*(A)\geq 1-s$, which is equivalent
to $\mu^*(A^c) \leq s$.

Assume now that $\mu^*(A^c)\leq s$, and we define a winning strategy $\tau$ for  $\II$ in $G(s,A)$. Suppose $\I$'s
first move is $(m_0,m_1)$, where $m_0+m_1>s$. Let $U$ be an open set containing $A^c$
with $\mu(U)<(m_0+m_1)$. Let $\tau (m_0,m_1)=x_0$ where $\mu(U\cap N_{x_0}) < m_{x_0}$. Continuing, we define $\tau$
so that for any play $p$ on length $n$ consistent with $\tau$ with $\II$'s moves $u=(x_0, \dots, x_{n-1})$,
we have that $\mu(U \cap N_u) < m_u< \mu(N_u)$. In particular, $N_u \cap U^c \neq \emptyset$. Since $U^c$
is closed, this shows that $\tau$ is a winning strategy for $\II$.

Finally, assume that $\mu_*(A^c)>s$. Let $F \subseteq A^c$ be closed with $\mu(F) >s$.
Fix $\epsilon >0$ such that $(1-\epsilon) \mu(F) >s$. We define a scaled measure $M$
and let $\sigma$ be the corresponding strategy for $\I$, that is,
if $(m_{u \concat 0}, m_{u \concat 1})$ is the response of $\sigma$ to $\II$ playing $u$,
then $m_{u \concat i}=M(u \concat i)$. Furthermore, we define $M$ so that all values $M(u)$ are
rational. First, we define $M(0)$, $M(1)$ to be rationals with
$(1-\epsilon) \mu(F\cap N_i) < M(i) < \mu(F\cap N_i)$. Note that $M(0)+M(1) >s$.
In general, if $M(u)\in \Q$ is defined with $(1-\epsilon) \mu(F\cap N_u) < M(u) < \mu(F\cap N_u)$,
let $M(u\concat 0)$, $M(u \concat 1)$ in $\Q$ be such that $M(u \concat 0)+M(u \concat 1)=M(u)$
and $(1-\epsilon) \mu(F\cap N_{u \concat i}) < M(u \concat i) < \mu(F\cap N_{u \concat i})$.
This defines the strategy $\sigma$ for $\I$ in $G(s,A)$. If $x \in 2^{\omega}$ is a run
according to $\sigma$, then for any initial segment $u$ of $x$ we have $M(u)>0$ and so
$\mu(F\cap N_u)>0$. As $F$ is closed, $x \in F \subseteq A^c$.

\end{proof}

\begin{corollary} [Martin \cite{Martin1998}]
If for all $s\in [0,1)$ the game $G(s,A)$ is determined, then $A$ is $\mu$-measurable.
In particular if $\det(\bG)$ holds for some pointclass $\bG$, then all sets
in $\bG$ are $\mu$-measurable. If $\ad$ holds, then all sets are $\mu$-measurable. 
\end{corollary}

Theorem~\ref{thm:equiv} and its proof have several quick corollaries which we next mention.

\begin{corollary}[Countable additivity]\label{lemma:countadd}
Let $\setof{\epsilon_n \suchthat n \in \omega}$ be such that $\epsilon_n\geq 0$ for all $n$ and
$\sum \epsilon_n = \epsilon < 1$.  Let $A_n \subseteq 2^\omega$ such that $\II$ has a
winning strategy in $G(\epsilon_n, A_n)$ for all $n$, then $\II$ has a winning
strategy in $G(\epsilon, A)$ where $A = \bigcap A_n$. 
\end{corollary}

\begin{proof}
From Theorem~\ref{thm:equiv} we have that $\mu^*(A_n^c) \leq \epsilon_n$ for each $n$.
So $\mu^*(A^c)=\mu^*(\bigcup_n A_n^c) \leq \epsilon= \sum \epsilon_n$. 
By Theorem~\ref{thm:equiv} again we have that $\II$ wins $G(\epsilon, A)$. 
\end{proof}

\begin{remark}
One can also give a direct game argument (avoiding Theorem~\ref{thm:equiv}) for Corollary~\ref{lemma:countadd}
using Lemma~\ref{lemma:rosendal}.
\end{remark}

The following corollary regards the symmetry of the game between the players.

\begin{corollary}\label{lemma:playerswap}
Let $\epsilon>0$ and $s>0$.
\[ \I~\text{wins}~G(s, A)  \Rightarrow \II~\text{wins}~G(1-s, A^c) \] 
\[ \II~\text{wins}~G(s-\epsilon, A)  \Rightarrow \I~\text{wins}~G(1-s, A^c)\]
\end{corollary}

\begin{proof}
For the first implication, $\I$ winning $G(s,A)$ is equivalent to $\mu_*(A^c)>s$, and
$\II$ winning $G(1-s,A^c)$ is equivalent to $\mu^*(A)\leq 1-s$. Since
$\mu_*(B)= 1-\mu^*(B^c)$ for any set $B$, the implication follows. The second implication
is done similarly. 
\end{proof}

\begin{remark}
Note that if $\I$ wins $G(0,A)$ then actually $\I$ wins $G(\epsilon, A)$ for some (all)
sufficiently small $\epsilon>0$, and so by Corollary~\ref{lemma:playerswap}
$\II$ wins $G(1-\epsilon, A)$ for some such positive $\epsilon$.
\end{remark}

And finally the following corollary gives the 
equivalence of the rational and real version of the game $G(s,A)$.

\begin{corollary} \label{cor:gameeq}
Let $A\subseteq 2^\omega$, and $s \in [0,1)$. Then the game $G(s,A)$ is equivalent to the game
$G'(s,A)$ which is played just as $G(s,A)$ except $\I$'s moves $m_p$ need not be rational.
\end{corollary}

\begin{proof}
If $\I$ has a winning strategy in $G(s,A)$ then trivially $\I$ has a winning strategy in $G'(s,A)$.
Likewise if $\II$ has a winning strategy in $G'(s,A)$, then trivially $\II$ has a winning strategy in
$G(s,A)$. If $\I$ has a winning strategy in $G'(s,A)$, then the proof of Theorem~\ref{thm:equiv} shows that
$\I$ has a winning strategy in $G(s,A)$.

Finally, suppose $\II$ has a winning strategy $\tau$ in $G(s,A)$. We get a strategy $\tau'$ for $\II$
in $G'(s,A)$ in a similar manner. If $\I$ makes move $(m'_{t\concat 0}, m'_{t \concat 1})$ in
$G'(s,A)$, then as in Theorem~\ref{thm:equiv} we get approximations $(m_{t\concat 0}, m_{t \concat 1})$
with $m_{t\concat 0}, m_{t \concat 1}  \in \Q$, with $(1-\epsilon) m'_{t\concat i} < m_{t\concat i}
<m'_{t \concat i}$ where $\epsilon>0$ is such that $(1-\epsilon) m'_\emptyset >s$. 
We have $\tau'$ play $\tau(m_{t\concat 0}, m_{t \concat 1})$.
Note here that if $m'_{t\concat i}=0$, then we also have $m_{t\concat i}=0$, so $\tau'$ is following the rules
of the game. Since every run of $\tau'$ has a corresponding run of $\tau$, $\tau'$ is a winning strategy for $\II$.
\end{proof}

\section{Borel-Cantelli and R\'{e}nyi-Lamperti} \label{sec:bc}

In this section we will first give a proof of the classical Borel-Cantelli
theorem using the measure game. We will then use the measure game to give a
new proof of the R\'{e}nyi-Lamperti theorem from probability theory, which
is a significant strengthening of one direction of the Borel-Cantelli theorem.
In this result, the hypothesis of mutual independence of the event is replaced by a
certain inequality. This result has found numerous recent applications to
number theory and related areas (c.f.\ \cite{harmon}).

We recall the statement of the Borel-Cantelli theorem. Let $(X,\sM,\mu)$ be a probability space
on the set $X$. Let $A_i \in \sM$ be $\mu$-measurable sets. One direction of Borel-Cantelli
asserts that if $\sum_i \mu(A_i) <\infty$ then $\mu(E)=1$, where
$E= \bigcup_k \bigcap_\ell A_\ell^c$ is the set of $x$ which are in only finitely
many of the $A_i$. This direction requires no assumption on the $A_i$.
The other direction of Borel-Cantelli asserts that if the $A_i$ are mutually independent
and $\sum_i \mu(A_i)=\infty$, then $\mu(E)=0$, that is, $\mu$-almost all $x$
are in infinitely many of the $A_i$. Recall the $A_i$ are {\em mutually independent}
if for all finite $F\subseteq \omega$, $\mu(\bigcap_{i \in F} A_i)=\prod_{i \in F} \mu(A_i)$.

The next theorem is the first direction of the Borel-Cantelli theorem.
We assume in the hypothesis that $\II$ wins the games $G(s_i,A_i)$, which for $\mu$-measurable
$A_i$ is equivalent to $\mu(A_i) \geq 1-s_i$, that is, $\mu(A_i^c)\leq s_i$.
We assume that $\sum_i s_i<\infty$ and deduce that $\II$
wins the game $G(0,\bigcup_{m=1}^\infty \bigcap_{i=m}^\infty A_i)$, which for measurable $A_i$
is equivalent to $\mu(\bigcup_{m=1}^\infty \bigcap_{i=m}^\infty A_i)=1$, that is,
almost all $x$ are in only finitely many of the $A_i^c$.

\begin{theorem}
Let $s_i \in [0, 1)$ such that $\sum_i s_i < \infty$. Let $\setof{A_i
    \suchthat i \in \omega}$ be a sequence of subsets of $2^\omega$
  such that $\II$ wins $G(s_i, A_i)$.  Then $\II$ wins $G(0,
  \bigcup_{m=1}^\infty \bigcap_{i=m}^\infty A_i)$
\end{theorem}

\begin{proof}
We will describe how to win $G(0, \bigcup_{m=1}^\infty
\bigcap_{i=m}^\infty A_i)$ as player $\II$.  Suppose $\I$ begins with
some move $(m_0,m_1)$ such that $m_0+m_1 = \epsilon>0$.  Since
$\sum_{i \in \omega} s_i < \infty$, there exists $n$ large enough so
that $\sum_{i \geq n} s_i < \epsilon$.  Since $\II$ wins $G(s_i, A_i)$
for any $i$, $\II$ wins in particular for $i \geq n$.  So by Corollary
\ref{lemma:countadd}, $\II$ wins $G\left(\sum_{i \geq n} s_i,
\bigcap_{i=n}^\infty A_i\right)$, say with some strategy $\tau$.
Since $m_0+m_1=\epsilon > \sum_{i \geq n} s_i$, we have that $(m_0,m_1)$ is a legal
first move for $\I$ in $G\left(\sum_{i \geq n} s_i,
\bigcap_{i=n}^\infty A_i\right)$. It
is not hard to see that if we play according to $\tau$ from this point
on, we will win $G(0, \bigcup_{m=1}^\infty \bigcap_{i=m}^\infty A_i)$.
\end{proof}

We now state the mutual independence hypothesis in terms of the measure game.
If $\II$ wins the games $G(s_i,A_i^c)$ for $i \in F$ ($F$ a finite subset of $\omega$)
then (for measurable $A_i$)
we have that $\mu(A_i^c) \geq 1-s_i$, that is, $\mu(A_i) \leq s_i$.
The definition requires that $\II$ wins
$G(\prod_{i \in F} s_i, \bigcup_{i \in F} A_i^c)$, which is equivalent to
$\mu(\bigcap_i A_i) \leq \prod_{i\in F} s_i$. The definition has a similar requirement
for player $\I$, which for measurable $A_i$ becomes 
$\mu(\bigcap_i A_i) \geq \prod_{i\in F} s_i$, and thus becomes the usual
definition. We note that for the proof of Theorem~\ref{thm:bcdiv} (the second direction
of Borel-Cantell) using the measure game, only the direction 
$\mu(\bigcap_i A_i) \leq \prod_{i\in F} s_i$ is required.

\begin{definition}
Let $\mathcal{I}$ be a set. Let $\setof{A_i \suchthat i \in
  \mathcal{I}}$ with $A_i \subseteq 2^\omega$ for all $i \in
\mathcal{I}$.  Then $\setof{A_i \suchthat i \in \mathcal{I}}$ are
\emph{mutually independent} if for any finite subset $F \subseteq
\mathcal{I}$, if player $P$ (either $\I$ or $\II$)  wins $G(s_i, A_i^c)$ for all $i \in F$,
then $P$ wins $G(\prod_{i \in F} s_i, \bigcup_{i \in F} A_i^c)$.
\end{definition}

The next theorem is the second direction of the Borel-Cantelli theorem stated in terms
of the measure game. The hypothesis that $\I$ wins the games $G(s_i,A_i)$, for measurable $A_i$,
gives that $\mu(A_i^c) > s_i$. We assume $\sum_i s_i=\infty$ and the $A_i$ are mutually independent,
and conclude that 
for any  $\epsilon>0$, $\I$ wins $G(1-\epsilon, \bigcup_{m=0}^\infty \bigcap_{i=m}^\infty A_i)$,
which gives that $\mu((\bigcup_{m=0}^\infty \bigcap_{i=m}^\infty A_i)^c)=1$, and thus
$\mu$-measure one many $x$ are in infinitely many of the $A_i^c$.

\begin{theorem} \label{thm:bcdiv}
  Suppose $\setof{A_i \suchthat i \in \omega}$ are mutually independent,
  and let $s_i \in [0, 1)$ such that $\sum_i s_i = \infty$.  Then if $\I$ wins $G(s_i, A_i)$ for all $i$, 
    then for any  $\epsilon>0$, $\I$ wins $G(1-\epsilon, \bigcup_{m=0}^\infty \bigcap_{i=m}^\infty A_i)$.
\end{theorem}

\begin{proof}
Note that since $\sum_{i \in \omega} s_i = \infty$ then $\prod_{i \in \omega} (1-s_i) = 0$.  By Corollary \ref{lemma:playerswap} $\II$ wins $G(1-s_i, A_i^c)$ for every $i$, and by the mutual independence of the sequence $\setof{A_i \suchthat i \in \omega}$, for any $L \leq M$, $\II$ wins $G\left(\prod_{i=L}^M (1-s_i), \bigcup_{i=L}^M A_i^c\right)$.  Let $\epsilon>0$ and choose a sequence of positive reals $\setof{\epsilon_k \suchthat k \in \omega}$ such that $\sum_{k \in \omega} \epsilon_k < \epsilon$.  Define sequences $\setof{L_k}$, $\setof{M_k}$ as follows:

Let $L_0 = 0$ and $M_0$ be large enough so that $\prod_{i=L_0}^{M_0} (1-s_i) < \epsilon_0$, and let $L_{k+1} = M_k + 1$ and $M_{k+1}$ be large enough so that $\prod_{i=L_{k+1}}^{M_{k+1}} (1-s_i) < \epsilon_{k+1}$.

Then we note that for any $k$, $\II$ wins $G\left(\prod_{i=L_{k}}^{M_{k}} (1-s_i), \bigcup_{i=L_k}^{M_k} A_i^c\right)$, and so $\II$ wins $G\left(\epsilon_k, \bigcup_{i=L_k}^{M_k} A_i^c\right)$.  So by Corollary \ref{lemma:countadd}, $\II$ wins $G\left(\sum_{k \in \omega} \epsilon_k, \bigcap_{k \in \omega}\bigcup_{i=L_k}^{M_k} A_i^c\right)$, and since $\sum_{k \in \omega} \epsilon_k < \epsilon$, by Corollary \ref{lemma:playerswap}, $\I$ wins $G\left(1-\epsilon, \bigcup_{k \in \omega}\bigcap_{i=L_k}^{M_k} A_i\right)$.

This implies $\I$ wins $G\left(1-\epsilon, \bigcup_{m=0}^\infty \bigcap_{i=m}^{\infty} A_i\right)$ since $\left(\bigcup_{k \in \omega}\bigcap_{i=L_k}^{M_k} A_i\right)^c \subseteq \left(\bigcup_{m=0}^\infty \bigcap_{i=m}^{\infty} A_i\right)^c$.
\end{proof}

We now show how the measure game can be used
to give a different, perhaps more constructive, proof of the R\'{e}nyi-Lamperti theorem.
A weaker form of the theorem, in which $\mu( \bigcap_n \bigcup_{i\geq n} A_i)$
is just asserted to be positive, was proved by
Erd\"{o}s and R\'{e}nyi \cite{erdos-renyi}. Lamperti \cite{lamperti}
also proved a somewhat different version of that theorem (with slightly different
hypothesis, but the same conclusion). The current stronger form of the theorem
was proved by Kochen and Stone \cite{kochen-stone} and Spitzer \cite{spitzer}.
Petrov \cite{petrov} gives a presentation in modern terminology and
proves other related results. 

Our proof is more constructive in the sense that it gives an explicit strategy
for getting into the desired set.
The version we prove is in the context of Borel probability measures on Polish spaces.
The restriction to Polish spaces seems necessary in order to obtain the more
constructive version we present. 
We recall the general statement of the theorem.

\begin{theorem}[R\'{e}nyi-Lamperti]  \label{lem:rl}
Let $(X,\sM,\mu)$ be a probability space and let $\{ A_i\}_{i \in \omega}$ be a sequence of
events $A_i\in \sM$ with $\sum_i \mu(A_i)=+\infty$ and
$\liminf_n \frac{\sum_{ i,j\leq n} \mu(A_i \cap A_j) }{(\sum_{i \leq n} \mu(A_i))^2} \leq D$,
then $\mu( \bigcap_n \bigcup_{i\geq n} A_i)\geq \frac{1}{D}$.
\end{theorem}

We note that if the $A_i$ are mutually independent then we may take $D=1$
in the hypothesis of the theorem, and thereby recover the
statement of the divergence part of the Borel-Cantelli theorem.

We prove two key lemmas, Lemmas~\ref{lem:side} and \ref{lem:approx}  which will enable us to apply the measure game.
Lemma~\ref{lem:side} will allow player $\II$ to pick an appropriate side in each round of the measure game,
and Lemma~\ref{lem:approx} will allow $\II$ to commit to meeting certain finite intersections
of the $A_i$ sets (while maintaining some hypotheses). 

First we show a simple combinatorial lemma we need for Lemma~\ref{lem:side}.

\begin{lemma} \label{lem:numsplit}
Suppose $\{ a_i\}_{i<n}$, $\{ b_i\}_{i<n}$, $\{ c_i\}_{i<n}$ are non-negative real numbers
with $b_i>0$ and $\sum_i c_i=1$. Then $\min_i \{ \frac{a_i}{b_i^2} c_i\}
\leq \frac{ \sum_i a_i}{( \sum_i b_i)^2}$.
\end{lemma}

\begin{proof}
If any of the $a_i$ are equal to $0$, then the result trivially holds,
so we assume $a_i>0$ for all $i$. 
Suppose without loss of generality the minimum occurs at $i=0$, that is $a_0 c_0 b_i^2 \leq a_i b_0^2 c_i$
for all $i$. 
Thus, $\frac{a_0 c_0 b_i^2}{a_i} \leq b_0^2 c_i$, and summing over $i\neq 0$ we get
\[
\sum_{i \neq 0} \frac{a_0 c_0 b_i^2}{a_i} \leq b_0^2 (\sum_{i \neq 0} c_i)=b_0^2 (1-c_0).
\]
Solving for $c_0$ we have $c_0 \leq \frac{b_0^2}{\sum_{i\neq 0} \frac{a_0 b_i^2}{a_i} +b_0^2}
=\frac{b_0^2}{\sum_i \frac{a_0 b_i^2}{a_i}}$.
We wish to show that $a_0 c_0 (\sum_i b_i)^2 \leq b_0^2 \sum_i a_i$. So, it suffices to show that
$a_0 b_0^2 (\sum_i b_i)^2 \leq  b_0^2 (\sum_i a_i) (\sum_i \frac{a_0 b_i^2}{a_i})$, that is
$(\sum_i b_i)^2 \leq (\sum_i a_i)(\sum_i \frac{b_i^2}{a_i})$. 
This is equivalent to showing $\sum_{i \neq j} b_i b_j \leq \sum_{i \neq j} \frac{ a_i b_j^2}{a_j}$,
that is $\sum_{i \neq j} (\frac{ a_i b_j^2}{a_j}-b_i b_j)\geq 0$. But we have
\begin{equation*}
\begin{split}
\sum_{i \neq j} (\frac{ a_i b_j^2}{a_j}-b_i b_j)&= \sum_{i \neq j} \frac{b_j}{a_j} (a_i b_j-a_j b_i)
\\ & = \sum_{i<j} (a_i b_j-a_j b_i) (\frac{b_j}{a_j} -\frac{b_i}{a_i})
\\ & =\sum_{i<j} \frac{(a_ib_j-a_jb_i)^2}{a_i a_j}
\\ & \geq 0.
\end{split}\end{equation*}

\end{proof}

\begin{lemma} \label{lem:side}
Let $(X,\sM,\mu)$ be a probability space and let $\{ B_i\}_{i \in \omega}$ be a sequence of
events $B_i\in \sM$ with $\sum_i \mu(B_i)=\infty$ and 
$\liminf_n \frac{\sum_{ i,j\leq n} \mu(B_i \cap B_j) }{(\sum_{i \leq n} \mu(B_i))^2} (1-m)<1$
for some $m <1$. 
Suppose $X=X_0 \cup X_1$ is a partition of the  set $X$ into $\mu$ measurable sets
and $m=m_0+m_1$ with $m_0,m_1\geq 0$.
Then for some $\ell \in \{ 0,1\}$ we have that
$\sum_i \mu(B_i \cap X_\ell)=\infty$ and
$\liminf_n \frac{\sum_{ i,j\leq n} \mu(B_i \cap B_j \cap X_\ell) }
{(\sum_{i \leq n} \mu(B_i\cap X_\ell))^2} (\mu(X_\ell)-m_\ell)<1$.
\end{lemma}

\begin{proof}
If for both $\ell=0$ and $\ell=1$ we have that $\sum_i \mu(B_i \cap X_\ell)=\infty$,
then we simply apply Lemma~\ref{lem:numsplit} as follows.
Let $n_0<n_1<\cdots$ be an increasing sequence such that
$\lim_k \frac{ \sum_{i,j \leq n_k} \mu(B_i\cap B_j)}{(\sum_{i <n_k} \mu(B_i))^2}
=\liminf_n \frac{ \sum_{i,j \leq n} \mu(B_i\cap B_j)}{(\sum_{i <n} \mu(B_i))^2}$.
For each $k$, apply Lemma~\ref{lem:numsplit} with $n=2$, $c_0=\frac{\mu(X_0)-m_0}{1-m}$,
$c_1=\frac{\mu(X_1)-m_1}{1-m}$, $a_0=\sum_{i,j \leq n_k} \mu(B_i\cap B_j\cap X_0)$,
$a_1=\sum_{i,j \leq n_k} \mu(B_i\cap B_j\cap X_1)$, $b_0=\sum_{i \leq n_k} \mu(B_i\cap X_0)$,
and $b_1=\sum_{i \leq n_k} \mu(B_i\cap X_1)$. This gives an $\ell_k\in \{ 0,1\}$
such that $\frac{a_{\ell_k}}{b_{\ell_k}^2} c_{\ell_k}
\leq \frac{ \sum_{i,j \leq n_k} \mu(B_i\cap B_j)}{(\sum_{i <n_k} \mu(B_i))^2}$.
So, $\frac{a_{\ell_k}}{b_{\ell_k}^2} (\mu(X_\ell)-m_\ell)
\leq \frac{ \sum_{i,j \leq n_k} \mu(B_i\cap B_j)}{(\sum_{i <n_k} \mu(B_i))^2}(1-m)<1$.
We may therefore fix an $\ell \in \{0,1\}$ such that $\ell_k=\ell$
for infinitely many $k$, and this $\ell$ satisfies the lemma.

Suppose now without loss of generality that $\sum_i \mu(B_i \cap X_0)<\infty$.
We show in this case that we have
$\liminf_n \frac{\sum_{ i,j\leq n} \mu(B_i \cap B_j\cap X_1) }
{(\sum_{i \leq n} \mu(B_i\cap X_1))^2} (\mu(X_\ell)-m_1)<1$. We consider two cases.
First suppose that $\sum_{i,j} \mu(B_i \cap B_j \cap X_0) =\infty$.
In this case $\lim_n  \frac{\sum_{ i,j\leq n} \mu(B_i \cap B_j\cap X_0) }
{(\sum_{i \leq n} \mu(B_i\cap X_0))^2} =\infty$. Thus, for large enough $k$,
the $\ell_k$ of the previous paragraph will all be equal to $1$,
and thus $\ell=1$ satisfies the lemma.
Suppose next that $\sum_{i,j} \mu(B_i \cap B_j \cap X_0) <\infty$.
In this case we have that
$\liminf_n \frac{\sum_{ i,j\leq n} \mu(B_i \cap B_j\cap X_1) }
{(\sum_{i \leq n} \mu(B_i\cap X_1))^2} (\mu(X_\ell)-m_\ell) < 1$.
If  $\sum_{i,j} \mu(B_i \cap B_j \cap X_1) <\infty$ then the $\liminf$ is
$0$ as the denominator tends to infinity [In fact it is true,
as is shown in the proof of the next lemma, that for all sequences of events that
the numerator is always at least as large as the denominator.  So
here the numerator must also tend to infinity. We don't need this
fact here, but it will be used in the main proof.] 
In the case where both the numerators and denominators tend to infinity,
this follows from the simple fact that
if $x_n$ and $z_n$ are bounded, $y_n$ and $w_n$ go to $\infty$, then
$\liminf_n \frac{x_n+y_n}{(z_n+w_n)^2}=\liminf_n \frac{y_n}{w_n^2}$. 
Thus we have that 
$\liminf_n \frac{\sum_{ i,j\leq n} \mu(B_i \cap B_j\cap X_1) }
{(\sum_{i \leq n} \mu(B_i\cap X_1))^2}=
\liminf_n \frac{\sum_{ i,j\leq n} \mu(B_i \cap B_j) }
{(\sum_{i \leq n} \mu(B_i))^2}$, and we also have $\mu(X_\ell)-m_\ell \leq 1-m$, and
the result follows.

\end{proof}

\begin{lemma} \label{lem:approx}
Let $(X,\sM,\mu)$ be a probability space and let $\{ A_i\}_{i \in \omega}$ be a sequence of
events $A_i\in \sM$ with 
$\liminf_n \frac{\sum_{ i,j\leq n} \mu(A_i \cap A_j) }{(\sum_{i \leq n} \mu(A_i))^2} \leq D$.
Then for all $\epsilon>0$ there is a $\delta>0$ such that if $C \in \sM$ satisfies
$\mu(C) < \delta$, then 
$\liminf_n \frac{\sum_{ i,j\leq n}
  \mu(A_i \cap A_j \cap (X\setminus C)) }{(\sum_{i \leq n} \mu(A_i \cap (X\setminus C)))^2}
\leq D+\epsilon$.
Furthermore, if $\sum_i \mu(A_i)=\infty$ then there exists $\delta>0$ such that
if $\mu(C) <\delta$ then $\sum_i \mu(A_i\cap (X\setminus C)) =\infty$. 
\end{lemma}

\begin{proof}
It is enough to show that for every $D>0$ and $\eta>0$ there is a $\delta>0$ such that
for any $C$ with $\mu(C)<\delta$ we have that if $\{ A_i\}$ is a finite sequence
of events satisfying 
$\frac{\sum_{ i,j\leq n} \mu(A_i \cap A_j) }{(\sum_{i \leq n} \mu(A_i))^2} \leq D$ 
then we have that
\[
\sum_i (\mu(A_i) \cap (X\setminus C)) \geq (1-\eta) \sum_i (\mu(A_i)).
\]
Since $\delta$ does not depend on $n$ or the $A_i$, this also
immediately gives the second statement.
Let $f=\sum_i \chi_{A_i}\colon X \to \R^{\geq 0}$. 
Note that $(\int f)^2 = (\sum_i \mu(A_i))^2$. Let $m=\int f$.
Then we have  $\int (f-m)^2= \int f^2 - m^2$, and so
$\sum_{i,j} \mu(A_i \cap A_j)= \int f^2 = m^2+ \int (f-m)^2$.
So, $\frac{\sum_{ i,j\leq n} \mu(A_i \cap A_j) }{(\sum_{i \leq n} \mu(A_i))^2}
=1+ \frac{ \int (f-m)^2}{m^2}\leq D$. So,
$ \frac{ \int (f-m)^2}{m^2} \leq D'=D-1$.
To prove the inequality it suffices (using $f= \sum_{i<n} \chi_{A_i}$)
to show the following  Lemma~\ref{lem:u1}.

\end{proof}

\begin{lemma} \label{lem:u1}
Let $D'>0$, then for every $\eta>0$, there is a $\delta>0$ such that
for every measurable function $f\colon X \to \mathbb{R}$ which satisfies
$\displaystyle \frac{ \int \left(f-\int f\right)^2}{\left(\int
f\right)^2}\leq D'$ and for every measurable set $C$ for which $\mu(C)
< \delta$, we have
\[\int_{X \setminus C} f \geq (1-\eta)\int_X f.\]
\end{lemma}

\begin{proof}
Fix $D', \eta >0$. Let $\delta>0$ be small enough so that
$\eta> \delta + \sqrt{\delta D' }$. Let $f$ and $C$ be as in the hypotheses,
so $\mu(C)<\delta$. Let $m=\int f$. 
Let $g\colon X \to \mathbb{R}$ be defined by
$g(x) = \begin{cases}f(x) & x \not \in C\\ \frac{\int_C f}{\mu(C)} & x \in C \end{cases}$.

Notice that $\int g = \int f = m$ and that $\int (g-m)^2 \leq \int (f-m)^2$.  To see this, observe that 
\[\int_C f^2 - g^2 = \int_C (g+(f-g))^2-g^2= \int_C 2g(f-g) + \int_C (f-g)^2=\int_C (f-g)^2\] 
since $g$ is constant on $C$ and $\int_C (f-g)=0$.  Using this and our hypothesis on $f$, we have
\[
\int_C (g-m)^2 \leq \int (g-m)^2 \leq \int (f-m)^2 \leq D'\, m^2
\]
On the other hand we have
\[
\int_C (g-m)^2= \mu(C) \left(\frac{\int_C f}{\mu(C)} -m \right)^2
=\frac{1}{\mu(C)} \left(\int_C f - m \mu(C)\right)^2.
\]
We thus have $m-\int_{X\setminus C} f -m \mu(C)=\int_C f-m \mu(C) \leq m \sqrt{D' \mu(C)}$. 
Hence
\[\int_{X\setminus C} f \geq m(1-\mu(C)- \sqrt{D' \mu(C)})
> (1-\eta)\int f
\]
by choice of $\delta$, since $\mu(C)<\delta$. 

\end{proof}

\begin{proof}[Proof of Theorem~\ref{lem:rl}]
Let $\{ A_i\}_{i \in \omega}$ and $D$ be as the statement of Theorem~\ref{lem:rl}.
For $0 \leq a \leq b$ let $A_{[a,b]}=\bigcup_{a\leq i \leq b} A_i$, and define $A_{[a,b)}$ similarly.
By choosing closed sets $F_i \subseteq A_i$ with $\mu(A_i\setminus F_i)$ sufficiently small,
we will have that 
$\liminf_n \frac{\sum_{ i,j\leq n} \mu(F_i \cap F_j) }{(\sum_{i \leq n} \mu(F_i))^2} \leq D$,
and so in proving the theorem we may assume that the sets $A_i$ are closed. 
Let $A= \bigcap_n \bigcup_{i\geq n} A_i$. To show $\mu(A)\geq \frac{1}{D}$, it
suffices to show that $\II$ wins the game $G(1-\frac{1}{D},A)$.

The idea for constructing $\II$'s strategy is to use Lemmas~\ref{lem:side}
and \ref{lem:approx}. At each stage of the game, we will first use Lemma~\ref{lem:side}
to pick the appropriate side which will maintain an inductive hypothesis as in the
statement of that lemma. We then use Lemma~\ref{lem:approx} to fix an interval of the sets
$\{ A_i\}$ which we commit to meeting at all future stages of the game. As the $A_i$
are closed, this will suffice to make sure the limit point we produce is in
infinitely many of the $A_i$. We now proceed with the definition of
$\II$'s strategy.

Say $\I$ makes first move
$m_0, m_1$ with $m_0 + m_1 > 1-\frac{1}{D}$, that is,
$D< \frac{1}{1-(m_0 + m_1)}=\frac{1}{\mu(N_\emptyset) -(m_0  + m_1)}$.
Let $m_\emptyset$ denote $m_0 + m_1$.
Thus we have
\[
\liminf_n \frac{ \sum_{i,j \leq n} \mu(A_i\cap A_j)}{(\sum_{i <n} \mu(A_i))^2}
(\mu(N_\emptyset)- m_\emptyset) <1.
\]

By Lemma~\ref{lem:side} there is a $t_0\in \{ 0,1\}$ such that
\begin{equation*} \label{eq:t1}
\liminf_n \frac{ \sum_{i,j \leq n} \mu(A_i\cap A_j\cap N_{t_0})}
{(\sum_{i \leq n} \mu(A_i\cap N_{t_0}))^2} (\mu(N_{t_0})-m_{t_0})
<1,
\end{equation*}
and such that

\begin{equation*} \label{eq:t2}
\sum_{i} \mu(A_i\cap N_{t_0})=\infty.
\end{equation*}

Since this is a strict inequality, we may then apply
Lemma~\ref{lem:approx} to pick $\ell_0$ large enough so that

\[
\liminf_n \frac{ \sum_{i,j \leq n} \mu(A_i\cap A_j\cap N_{t_0}\cap A_{[0,\ell_0)})}
{(\sum_{i \leq n} \mu(A_i\cap N_{t_0}\cap A_{[0,\ell_0)}))^2} (\mu(N_{t_0})-m_{t_0}) <
1,
\]
and also so that $\sum_i \mu (A_i \cap N_{t_0}\cap A_{[0,\ell_0)}) = \infty$.
Here we have used Lemma~\ref{lem:approx} with $X=N_{t_0}$ and $\mu$ is normalized
$\mu$ measure on $N_{t_0}$, and $C= A_{[0,\infty)}\setminus A_{[0,\ell_0]}$.

Since $\sum_i \mu(A_i\cap N_{t_0}\cap A_{[0,\ell_0)}) =\infty$ it is easy to check that we still have the inequality for the sequence of events $\setof{A_i \suchthat i \geq \ell_0}$, i.e.
\[
\liminf_n \frac{ \sum_{\ell_0\leq i,j \leq n} \mu(A_i\cap A_j\cap N_{t_0}\cap A_{[0,\ell_0)})}
{(\sum_{\ell_0\leq i \leq n} \mu(A_i\cap N_{t_0}\cap A_{[0,\ell_0)}))^2} (\mu(N_{t_0})-m_{t_0}) <
1.
\]

$\II$'s response to $\I$'s move of $m_0, m_1$ will be to play
$t_0$ unless $m_{t_0}=0$, in which case player $\II$ must follow the rules and play $t'_0=1-t_0$.
We check in this case that the we can still maintain these conditions using $t'_0$.
 Assume without loss of generality that $t_0=1$, that is, $m_1=0$,
so $m_\emptyset=m_0$. 

Consider the expression.
\[
\liminf_n \frac{ \sum_{i,j \leq n} \mu(A_i\cap A_j\cap N_{1})}
{(\sum_{i \leq n} \mu(A_i\cap N_{1}))^2} (\mu(N_{1})-m_1)=
 \liminf_n \frac{ \sum_{i,j \leq n} \mu(A_i\cap A_j\cap N_{1})}
{(\sum_{i \leq n} \mu(A_i\cap N_{1}))^2} \mu(N_1).
\]
Now note that we have that $\liminf_n \frac{ \sum_{i,j \leq n} \mu(A_i\cap A_j\cap N_{1})}
{(\sum_{i \leq n} \mu(A_i\cap N_{1}))^2} \mu(N_1) \geq 1$, since for any measure $\nu$ and sequence of events $A_i$, the quantity
$\liminf_n \frac{ \sum_{i,j \leq n} \nu(A_i\cap A_j)}
{(\sum_{i \leq n} \nu(A_i))^2}$ is always at least $1$.  So, 
from Lemma~\ref{lem:side} that
\[
\liminf_n \frac{ \sum_{i,j \leq n} \mu(A_i\cap A_j\cap N_{0})}
{(\sum_{i \leq n} \mu(A_i\cap N_{0}))^2} (\mu(N_{0})-m_0)<1
\]
and $\sum_i \mu(A_i \cap N_0) =\infty$, 
and we are done.

For the general step, suppose $t \in 2^{<\omega}$, and we just finished the round of the game where
$\I$ has played $m_{t\res (|t|-1)\concat 0}$ and $m_{t\res (|t|-1)\concat 1}$
and we have defined $\II$'s response by picking $t$ and thus
$N_t$ and $m_t$. We assume inductively that we have defined a sequence $\ell_0<\ell_1<\cdots< \ell_{|t|}$
and we have
\begin{equation} \label{eqn:inha}
\liminf_n \frac{\sum_{\ell_{|t|} \leq i,j \leq n} \mu(A_i \cap A_j \cap N_t \cap
\bigcap_{k<|t|}A_{[\ell_k,\ell_{k+1})})}
{(\sum_{\ell_{|t|}\leq i \leq n} \mu(A_i \cap N_t \cap
\bigcap_{k<|t|}A_{[\ell_k,\ell_{k+1})})^2} (\mu(N_t)-m_t)
< 1.
\end{equation}
and also
\begin{equation} \label{eqn:inhb}
\sum_i  \mu(A_i \cap N_t \cap \bigcap_{k<|t|}A_{[\ell_k,\ell_{k+1})}) =\infty.
\end{equation}
Suppose now that player $\I$ plays $m_{t\concat 0}$ and $m_{t\concat 1}$ where 
$m_t=m_{t\concat 0}+ m_{t\concat 1}$. Suppose first that $m_{t\concat 0}, m_{t\concat 1}>0$.
As in the first step, we first apply Lemma~\ref{lem:side} to get a  $b \in \{ 0,1\}$ such that if we set
$t'=t\concat b$ then using $N_{t'}$, $m_{t'}$, and intersecting with the set
$\bigcap_{k<|t|}A_{[\ell_k,\ell_{k+1})}$
in both the numerator and denominator, we maintain iequalities~\ref{eqn:inha} and \ref{eqn:inhb}.

We then apply Lemma~\ref{lem:approx} to get an $\ell_{|t'|}= \ell_{|t|+1}$
so that we maintain inequalities~\ref{eqn:inha} and \ref{eqn:inhb},  now intersecting also with
$A_{[\ell_{|t|}, \ell_{|t'|})}$. 

Let $\II$ respond by playing $b$.
This completes the definition of $\II$'s strategy $\tau$ in
the game $G(1-\frac{1}{D},A)$. Note that we have maintained the inductive
inequalities~\ref{eqn:inha} and \ref{eqn:inhb} at all positions of the game.

To see that this is a winning strategy for $\II$, let $\vec{t}$ be the moves by $\II$
in a run of the game consistent with $\tau$. Let $\{ x\}= \bigcap_n N_{\vec{t}\res n}$.
For all $n$ we have
\[
\sum_{i \geq \ell_n}  \mu(A_i \cap N_{\vec{t}\res n} \cap \bigcap_{k<n} A_{[\ell_k, \ell_{k+1})} )=\infty.
\]
In particular we have that for every $n$ that
$\mu (N_{\vec{t}\res n} \cap \bigcap_{k<n} A_{[\ell_k, \ell_{k+1})} ) >0$, since these sets are descending.
For each $k$, we thus have that for all $n$ that $\mu( N_{\vec{t}\res n}\cap A_{[\ell_k,\ell_{k+1})}) >0$,
and so $N_{\vec{t}\res n}\cap A_{[\ell_k,\ell_{k+1})}\neq \emptyset$. Since
$A_{[\ell_k,\ell_{k+1})}$ is a closed set, this gives that $x \in A_{[\ell_k,\ell_{k+1})}$.
Since this holds for all $k$, we have that $x \in A$, and thus $\II$ has won the run of the game.

\end{proof}

\section{Fubini's Theorem} \label{sec:fubini}

In this section, we prove Fubini's theorem.
\begin{theorem} \label{thm:fubini}
Let $\mu$ be a Borel probability measure on the Polish space $X$.
Let $A \subseteq X \times X$ be $\mu \times \mu$-measurable.
Then the following are equivalent:
\begin{enumerate}
\item $(\mu \times \mu)(A)=0$ \label{f1}
\item For almost all $x \in X$, $\mu(A_x)=0$ where $A_x = \setof{y \in X \suchthat (x, y) \in A}$ \label{f2}
\item For almost all $y \in X$, $\mu(A^y)=0$ where $A^y = \setof{x \in X \suchthat (x, y) \in A}$ \label{f3}
\end{enumerate}
\end{theorem}

In order to prove this theorem using the measure game, we must use a two-dimensional version of the game.
We give the precise definition of the game.

\begin{definition}
Let $s \in \left[0, 1\right)$ and let $A \subseteq 2^\omega \times 2^\omega$, we define a game $G_2(s, A)$ similarly to the measure game $G$, except at each move player $\I$ specifies four numbers instead of two (indexed by $\setof{0, 1} \times \setof{0, 1})$.
\begin{center}
\begin{games}
\game[player spacing=0.25cm, game label={$G_2(s, A)$}, dots]{
    {$m_{(0,0)}, m_{(0,1)}, m_{(1, 0)}, m_{(1, 1)}$}{$x_0$}
    {$m_{x_0 \concat(0,0)}, m_{x_0 \concat(0,1)}, m_{x_0 \concat(1,0)}, m_{x_0 \concat(1,1)}$}{$x_1$}
    }
\end{games}
\end{center}

We require that player $\I$ must follow the rules:
\begin{itemize}
\item $\forall t \in (2\times 2)^{<\omega},~ m_t \in \mathbb{Q}$
\item $m_{(0,0)} + m_{(0,1)} + m_{(1, 0)} + m_{(1, 1)} > s$
\item $\forall t \in (2\times 2)^{<\omega},~ 0\leq m_t \leq (\mu \times \mu)(N_{t_0} \times N_{t_1})$
\item $\forall t \in (2\times 2)^{<\omega},~0< (\mu \times \mu)(N_{t_0} \times N_{t_1}) \Rightarrow m_{t} < (\mu \times \mu)(N_{t_0} \times N_{t_1})$
\item $\forall t \in (2\times 2)^{<\omega}, ~ m_{t \concat(0,0)}+ m_{t \concat(0,1)}+ m_{t \concat(1,0)}+ m_{t \concat(1,1)}= m_t$
\end{itemize}
and player $\II$ must follow the rules:
\begin{itemize}
\item $\forall n \in \omega,~  x_n \in \setof{0, 1} \times \setof{0, 1}$
\item $\forall n \in \omega,~  m_{x_0 x_1 \dots x_n} \neq 0$.
\end{itemize}
\end{definition}

\begin{remark}
We will need to analyze strategies for player $\II$ in the game $G_2$, and so we remark that Lemma~\ref{lemma:delta} also applies to this game.
\end{remark}

\begin{theorem} \label{thm:fub1}
Let $A \subseteq 2^\omega \times 2^\omega$.  If $\II$ wins $G_2(0, A)$, then $\II$ wins $G(0, B)$ where \[B = \setof{x \in 2^\omega \suchthat \II~\text{wins}~G(0, A_x)}\]
\[A_x = \setof{y \in 2^\omega \suchthat (x, y) \in A}\]
\end{theorem}
\begin{proof}
By Corollary \ref{lemma:countadd}, it suffices to show that for any $\epsilon>0$,  $\II$ wins $G(0, B_\epsilon)$ where\\ $B_\epsilon =\setof{x \in 2^\omega \suchthat \II~\text{wins}~G(\epsilon, A_x)}$. With this in mind, fix $\epsilon>0$ and let $\tau$ be a winning strategy for $\II$ in $G_2(0, A)$.  We will define a strategy for player $\II$ in the game $G(0, B_\epsilon)$, which in each round will play moves to minimize a certain quantity.  This quantity will be computed by analyzing the given strategy $\tau$ and assigning values to each of the four quadrants which are available to $\tau$ at each round of $G_2(0, A)$.  While defining this strategy, we will also define a tree of positions in the two-dimensional game which are consistent with $\tau$, in which the first coordinate of every node agrees with the moves $x_0, \dots x_k$ we are playing in $G(0, B_\epsilon)$.

Consider first $\delta_\tau^\emptyset$, in the notation of Lemma~\ref{lemma:delta} (Note that because $s=0$, $\delta_\tau^\emptyset(i, j) = 0$ for each $(i, j) \in 2 \times 2$). Suppose $\I$ plays $m_0$ in $G(0, B_\epsilon)$.  Since on this move, $\delta_\tau^\emptyset$ is trivial, we play $x_0$ to maximize $m_{x_0}$.  For each $j \in 2$, if $\delta_\tau^\emptyset(x_0, j)=(\mu \times \mu)(N_{x_0} \times N_{j})$, (i.e. they are both zero) we set $q_j=(\mu \times \mu)(N_{x_0} \times N_{j})=0$ and declare the node $j$ to be ``dead''.  Otherwise if $\delta_\tau^\emptyset(x_0, j)<(\mu \times \mu)(N_{x_0} \times N_{j})$, choose $q_j$ such that $\delta_\tau^\emptyset(x_0, j)<q_j<(\mu \times \mu)(N_{x_0} \times N_{j})$ and
\[\frac{q_0 + q_1}{m_{x_0}} < \epsilon.\]
For each live node $j$, we fix a move $\hat{q}_j$ by $\I$ which assigns $q_j$ to the quadrant $(x_0, j)$ with $\tau$ selecting that quadrant.  For simplicity of notation, we do not keep explicit track of the values assigned to the other quadrants by $\hat{q}_j$.  

For each live node $j$, we have the $4$-tuple $\delta_\tau^{\hat{q}_j}$, we define $\delta_j=\delta_\tau^{\hat{q}_j}$ in this case. For each dead node $j$, we define $\delta_j(i, k)=(\mu \times \mu)(N_{x_0 \concat i} \times N_{j \concat k})$. Note that at least one node must be alive at this step (and at each later step). 

Next suppose $\I$ plays $(m_{x_0 \concat 0}, m_{x_0 \concat 1})$ in response to our move $x_0$ in $G(0, B_\epsilon)$. We define our strategy's response by choosing $x_1$ (which determines a vertical section on which to continue play)
to minimize $\frac{\delta_0(x_1, 0) + \delta_0(x_1, 1) + \delta_1(x_1, 0) + \delta_1(x_1, 1)}{m_{x_0 x_1}}$, which will guarantee that 
\[\frac{\delta_0(x_1, 0) + \delta_0(x_1, 1) + \delta_1(x_1, 0) + \delta_1(x_1, 1)}{m_{x_0 x_1}} \leq \frac{q_0+q_1}{m_{x_0}}<\epsilon.\]
This is because for dead nodes, $\sum_{i, k \in 2}\delta_\tau^{\hat{q_j}}(i, k) = q_j = (\mu \times \mu)(N_{x_0 \concat i} \times N_{j \concat k})$, while for live nodes we have $\sum_{i, k \in 2}\delta_\tau^{\hat{q_j}}(i, k) \leq q_j$ by Lemma~\ref{lemma:delta}.  We also use that $\min\setof{\frac a c, \frac b d} \leq \frac{a+b}{c+d}$ for any non-negative $a, b, c, d$ with $c+d>0$.  
For $p=j \concat k$, if $j$ is a dead node, then $p$ is declared to be dead.  We also declare $p$ to be a dead node if $\delta_j(x_1, k) = (\mu \times \mu)(N_{x_0 \concat x_1} \times N_{j \concat k})$, otherwise $p$ is alive.

Again, for each dead node $p = j \concat k$, set $q_p = \delta_j(x_1, k) = (\mu \times \mu)(N_{x_0 \concat x_1} \times N_{j \concat k})$.  For the live nodes, choose $q_p$ such that $\delta_j(x_1, k) < q_p < (\mu \times \mu)(N_{x_0 \concat x_1} \times N_{j \concat k})$ and so that
\begin{equation}\label{eq:fubindineq}
\frac{\sum_{p \in 2^2}q_p}{m_{x_0 x_1}} <\epsilon.
\end{equation}
and such that there is some move $\hat{q}_p$ by $\I$ which assigns $q_p$ to $(x_1, k)$ with $\tau$ selecting that quadrant. 

In general, suppose we've moved $x_0 \dots x_{n-1}$ and that we've defined $q_p$ for $p \in 2^n$. The node $p$ is live if $q_p< (\mu \times \mu)(N_{x_0 \dots x_{n-1}} \times N_p)$.  For each live node $p$, we have a corresponding $4$-tuple $\delta^{\hat{q}_p}_\tau$ and we set $\delta_p =\delta^{\hat{q}_p}_\tau$.  For dead nodes $p \in 2^n$, we set $\delta_p(i, j)= (\mu \times \mu)(N_{x_0 \dots x_{n-1} \concat i} \times N_{p \concat j})$.  Suppose $\I$ plays $(m_{x_0 \dots x_{n-1} \concat 0}, m_{x_0 \dots x_{n-1} \concat 1})$ We choose $x_n$ to minimize the quantity 
\[\frac{\sum_{p \concat j \in 2^{n+1}} \delta_p(x_n, j)}{m_{x_0 \dots x_n}}.\]
By induction we have maintained 
\[\frac{\sum_{p \in 2^n} q_p}{m_{x_0 \dots x_{n-1}}} < \epsilon\] 
and so we have
\[\frac{\sum_{p \concat j \in 2^{n+1}} \delta_p(x_n, j)}{m_{x_0 \dots x_n}} \leq \frac{\sum_{p \in 2^n} q_p}{m_{x_0 \dots x_{n-1}}} < \epsilon.\]
We declare $p \concat j$ to be dead if $\delta_p(x_n, j) = (\mu \times \mu)(N_{x_0 \dots x_{n}} \times N_{p \concat j})$, in which case, we set $q_{p \concat j} =(\mu \times \mu)(N_{x_0 \dots x_{n}} \times N_{p \concat j})$.  For live nodes we choose $q_{p \concat j}$ with $\delta_p(x_n, j) <q_{p \concat j} <(\mu \times \mu)(N_{x_0 \dots x_{n}} \times N_{p \concat j})$ to maintain
\[\frac{\sum_{p \concat j \in 2^{n+1}} q_{p \concat j}}{m_{x_0 \dots x_n}} < \epsilon.\] such that there is a move $\hat{q}_{p \concat j}$ which assigns $q_{p \concat j}$ to $(x_n, j)$ for which $\tau$ selects $(x_n, j)$.  This finishes the definition our strategy.

We aim to show that any result $x$ consistent with our strategy is a win for $\II$.  Let $T = \setof{\emptyset} \cup \setof{p \in 2^{<\omega} \suchthat q_p < (\mu \times \mu)(N_{x_0 \dots x_{n-1}} \times N_{p})}$ i.e. $T$ consists of all live nodes.

Note that if $p\not \in T$, then $q_p = (\mu \times \mu)(N_{x_0 \dots x_{n-1}} \times N_{p})$.  Note also that, by the rules of the game, $m_{x_0 \dots x_{n-1}}< \mu(N_{x_0 \dots x_{n-1}})$, and thus 
\[\sum_{p \in 2^{n} \setminus T} q_p = \sum_{p \in 2^{n} \setminus T} (\mu \times \mu)(N_{x_0 \dots x_{n-1}} \times N_{p}) = \mu(N_{x_0 \dots x_{n-1}}) \sum_{p \in 2^{n} \setminus T} \mu(N_{p})\]
but on the other hand we maintained 
\[\sum_{p \in 2^{n}} q_p < \epsilon m_{x_0 \dots x_{n-1}} < \epsilon \mu(N_{x_0 \dots x_{n-1}}).\]
thus
\[ \mu(N_{x_0 \dots x_{n-1}}) \sum_{p \in 2^{n} \setminus T} \mu(N_{p}) =\sum_{p \in 2^{n} \setminus T} q_p \leq \sum_{p \in 2^{n}} q_p < \epsilon \mu(N_{x_0 \dots x_{n-1}})\]
and so
\[\sum_{p \in 2^{n} \setminus T} \mu(N_{p}) < \epsilon\]
and therefore $\mu([T]^c) \leq \epsilon$. So by Theorem \ref{thm:equiv}, $\II$ wins $G(\epsilon, [T])$. Since $\tau$ was a winning strategy for $\II$ in $G_2(0, A)$, we have $[T] \subseteq A_x$, and so $\II$ wins $G(\epsilon, A_x)$.  Thus, $x \in B_\epsilon$ and so the strategy we've defined wins $G(0, B_\epsilon)$.

\end{proof}

The next theorem is the symmetrical version of Theorem~\ref{thm:fub1} for player $\I$. It says that if player $\I$
can win the two-dimensional game $G_2$ for a certain set, then player $\I$ can win the one-dimensional
game $G$ to produce an $x \in 2^\omega$ such that player $\I$ can win the game $G(0, A_x)$ to produce a $y$
in the $x$-section of the set.

\begin{theorem} \label{thm:fub2}
Let $A \subseteq 2^\omega \times 2^\omega$.  Suppose $\I$ wins $G_2(0, A)$, then $\I$ wins $G(0, B)$  where \[B = \setof{x \in 2^\omega \suchthat \I~\text{doesn't win}~G(0, A_x)}\]
\[A_x = \setof{y \in 2^\omega \suchthat (x, y) \in A}.\]
\end{theorem}

\begin{proof}
Suppose $\I$ wins $G_2(0, A)$, then by Corollary~\ref{lemma:playerswap},
$\II$ wins $G_2(1-\epsilon, A^c)$ for some $\epsilon >0$.  Note that
also by Corollary~\ref{lemma:playerswap}, it suffices to show that for
some $\gamma>0$, $\II$ wins $G(1-\gamma, B^c)$.

Assume $\tau$ is a winning strategy for $\II$ in $G_2(1-\epsilon, A^c)$.  Let $\gamma<\epsilon$ and let $\beta>0$ be such that $1-\gamma$ = $\frac{1-\epsilon}{1-\beta}$.  We will show that for any such $\gamma$, $\II$ wins $G(1-\gamma, B^c)$.

Suppose $\I$ plays $(m_0, m_1)$ in $G(1-\gamma, B^c)$, and let $\delta^\emptyset_\tau$ be as in the notation of Lemma~\ref{lemma:delta}.  Note that 
\[
\sum_{i, j \in 2} \delta^\emptyset_\tau(i, j) \leq 1-\epsilon =
(1-\gamma)(1-\beta) < \left(m_0 + m_1\right)(1-\beta)
\ <\mu(N_\emptyset)(1-\beta).
\]

This implies that there is some $x_0$ we can play so that 
\[\sum_{j \in 2} \delta^\emptyset_\tau(x_0, j) < m_{x_0}(1-\beta) < \mu(N_{x_0})(1-\beta).\]

As in the proof of Theorem~\ref{thm:fub1}, we declare the length one sequence $j$ to be dead
if $\delta^\emptyset_\tau(x_0, j) = (\mu \times \mu)(N_{x_0} \times N_j)$,
otherwise $j$ is declared to be live.
For each $j\in 2$ which is dead, set $q_j=\delta^\emptyset_\tau(x_0, j)=(\mu\times \mu)(N_{x_0}\times N_j)$.
For each $j\in 2$ which is live, we can choose moves $\hat{q}_j$ by $\I$ which assign values $q_j$ to the quadrant $(x_0, j)$ such that 
\[
\delta^\emptyset_\tau(x_0, j) < q_j < (\mu \times \mu)(N_{x_0} \times N_j),\] and
\[q_0 + q_1 < m_{x_0}(1-\beta),
\]
and also so that $\tau$ selects the quadrant $(x_0, j)$ in response to $\hat{q}_j$. 
This finishes the definition of first step of the strategy.

Suppose $\I$ plays $(m_{x_0 \concat 0}, m_{x_0 \concat 1})$ in response to our move $x_0$.
For each live node $j$ we have the $4$-tuple $\delta^{\hat{q}_j}_\tau$ from Lemma~\ref{lemma:delta}.
We set $\delta_j=\delta^{\hat{q}_j}_\tau$.
For $j$ a dead node, as in Theorem~\ref{thm:fub1} we set $\delta_j(i,k)=(\mu \times \mu)(N_{x_0\concat i}\times
N_{j \concat k})$.

By Lemma~\ref{lemma:delta} we have
\[
\sum_{i, j, k \in 2} \delta_j(i,k) \leq \sum_{j \in 2} q_j < m_{x_0}(1-\beta) =
\left(m_{x_0 \concat 0} + m_{x_0 \concat 1}\right)(1-\beta)
\]
and thus there exists a move $x_1$ we can play to maintain 
\[
\sum_{j, k \in 2} \delta_j(x_1, k)  <  m_{x_0 x_1}(1-\beta) < \mu(N_{x_0 x_1})(1-\beta).
\]
For $t=j \concat k\in 2^2$, we declare $t$ to be dead if either $j$ is dead or if
$\delta_j(x_1,k)=(\mu\times \mu)(N_{x_0 x_1}\times N_{j\concat k})$. If $t$
is a dead node, we set $q_t=(\mu\times \mu)(N_{x_0 x_1}\times N_{j\concat k})$.
For each live node $t$ we choose $q_t$ such that
$\delta^{\hat{q}_j}_\tau (x_1, k)< q_t<(\mu \times \mu)(N_{x_0 x_1} \times N_{t})$
and such that
\[
\sum_{t \in 2^2} q_t < m_{x_0 x_1} (1-\beta).
\]
For the live nodes $t$ we also pick moves $\hat{q}_t$ by $\I$  which
each assign the value $q_t$ to the corresponding quadrant and such that $\tau$
will respond to $\hat{q}_t$ by picking $(x_1,k)$. This ends the second step of the
construction.

Now similarly to the first step, for each $t=j \concat k \in 2^2$, we define $q_{t}$ by either $q_{t}= \delta^{\hat{q}_j}_\tau (x_1, k)=(\mu \times \mu)(N_{x_0 x_1} \times N_{t})$, or $\delta^{\hat{q}_j}_\tau (x_1, k)< q_t<(\mu \times \mu)(N_{x_0 x_1} \times N_{t})$ and in this case there is an associated move by $\I$, $\hat{q}_t$, for which $\tau$ selects $(x_1, k)$.  We also maintain that \[\sum_{t \in 2^2} q_t < m_{x_0 x_1} (1-\beta),\]
this finishes the second step of the induction.

In general, we maintain the inequality
\[
\sum_{t \in 2^n} q_t < m_{x_0 \dots x_{n-1}} (1-\beta)
\]
and for each $t \in 2^n$ we have that either $q_t$ is live, meaning
$q_t< (\mu \times \mu)(N_{x_0 \dots x_{n-1}}\times N_{t})$, in which case
we have a move $\hat{q}_t$ of $\I$ in $G_2$, or else $t$ is dead
and $q_t=(\mu \times \mu)(N_{x_0 \dots x_{n-1}}\times N_{t})$.

This finishes the definition of our strategy, and if we have some real
$x \in 2^\omega$ which is consistent with this strategy, then we have
also the numbers $q_t$ for every $t\in 2^{<\omega}$, defined by the
construction above.  We define $T$ again to be the tree of live nodes.

We maintained through the construction of the strategy that 
\[
\sum_{t \in 2^{n}} q_{t} < m_{x_0 \dots x_{n-1}} (1-\beta) < \mu(N_{x_0 \dots x_{n-1}}) (1-\beta).
\]
Note that for any node $t \in 2^{n} \setminus T$ we have $q_{t} = (\mu \times \mu) (N_{x_0 \dots {x_{n-1}}} \times N_{t})$, and thus
\[
\sum_{t \in 2^{n} \setminus T} q_{t} = \mu(N_{x_0 \dots x_{n-1}}) \sum_{t \in 2^{n} \setminus T} \mu(N_{t})
\]
and so
\[
\sum_{t \in 2^{n} \setminus T} \mu(N_{t}) \leq (1-\beta).
\]

Let $x$ be according to the strategy defined above, and let $T$ be the
corresponding tree.  We claim that $\I$ has a winning strategy in
$G(0, A_x)$, so in particular, $x \in B^c$.  Note that because $\tau$
is a winning strategy in $G_2(1-\epsilon, A^c)$, $[T] \subseteq
A_x^c$, and since for all $n$ we have $ \sum_{t\in 2^{n} \setminus T}
\mu(N_{t}) \leq (1-\beta)$, it follows that  $\mu([T]) \geq \beta >0$. By
Theorem~\ref{thm:equiv}, $\I$ wins $G(0, [T]^c)$.
\end{proof}

It is straightforward using Theorem~\ref{thm:equiv}
to see that Theorems~\ref{thm:fub1} and \ref{thm:fub2} imply
Theorem~\ref{thm:fubini}. Theorem~\ref{thm:fub1} immediately gives that (\ref{f1})
of Theorem~\ref{thm:fubini} implies (\ref{f2}) and (\ref{f3}) of \ref{thm:fubini}.
Similarly, Theorem~\ref{thm:fub2} gives that (\ref{f2}) implies (\ref{f1})
of \ref{thm:fubini}.

\section{The Unfolded Measure Game} \label{sec:unfolding}

The Banach-Mazur game, corresponding to category, has an ``unfolded version'' which
has many applications, such as the Baire property for analytic sets and continuous
uniformizations on comeager sets. We introduce here an unfolded version of the measure game
and show it can be used to direct proofs of the corresponding facts for measure.

\begin{definition}
The \emph{unfolded measure game} is the following game.
Let $F \subseteq 2^\omega \times \omega^\omega$, and $s \in [0, 1)$.  Let $A=\setof{x \in 2^\omega \suchthat \exists y \in \omega^\omega~(x, y) \in F}$ the game $G_{\text{unfolded}}(s, F)$ is played just like the measure game $G(s, A)$, except that player $\II$ must, in addition to producing $x \in 2^\omega$ such that there is some $y$ with $(x, y) \in F$, also produce such a $y$, by playing it digit by digit.  We do not require $\II$ to play a digit of $y$ during every round, but $\II$ must do so infinitely often, as the diagram below illustrates.
\begin{center}
\begin{games}
\game[game label=${G_{\text{unfolded}}(s, F)}$, move spacing=-0.25 cm]
{
{$m_0, m_1$}{$x_0$}
{$\dots$}{$\dots$}
{${m_{p_0 \concat 0}, m_{p_0 \concat 1}}$}{${x_{k_0}, y_0}$}
{$\dots$}{$\dots$}
{${m_{p \concat 0}, m_{p  \concat 1}}$}{$x_n$}
{$\dots$}{$\dots$}
{${m_{p_i \concat 0}, m_{p_i \concat 1}}$}{${x_{k_i}, y_i}$}
}
\end{games}
\end{center}

Player $\I$ must follow the rules:
\begin{itemize}
\item $\forall t \in 2^{<\omega},~ m_t \in \mathbb{Q}$ and $m_t\geq0$
\item $m_0+m_1 > s$
\item $\forall t \in 2^{<\omega},~ 0< \mu(N_t) \Rightarrow m_{t} < \mu(N_t)$
\item $\forall t \in 2^{<\omega}, ~ m_{t\concat 0}+m_{t \concat 1}= m_t$
\end{itemize}
and player $\II$ must follow the rules:
\begin{itemize}
\item $\forall n \in \omega, ~ x_n \in 2$
\item $\forall n \in \omega, ~ m_{x_0, \dots x_n} \neq 0$
\item $\forall i \in \omega, ~ \exists k_i $ such that on his $k_i$th move, $\II$ plays $x_{k_i}, y_i$
where $y_i \in \omega$
\end{itemize}

If both players follow the rules, then 
$\II$ wins the run of $G_{\text{unfolded}}(s, F)$ if and only if $(x, y) \in F$.

Note that the last rule for $\II$ is not a closed requirement for $\II$ and thus not a ``rule''
in the usual sense (it cannot be determined at any finite stage if $\II$ has followed
this requirement).
\end{definition}

The importance of the unfolded game is illustrated by the following theorem.
\begin{theorem} \label{theorem:unfolding}
Let $F \subseteq 2^\omega \times \omega^\omega$ and
$A=\setof{x \in 2^\omega \suchthat \exists y \in \omega^\omega~ (x, y) \in F}$.
Then if  $~\I$ \emph{(}or $\II$\emph{)} wins $G_{\text{unfolded}}(s, F)$ then
$\I$ \emph{(}resp.\ $\II$\emph{)} wins $G(s, A)$.
\end{theorem}

To prove this theorem, we first note that the case where $\II$ has a winning strategy in
$G_{\text{unfolded}}(s, F)$ immediately gives a winning strategy for $\II$ in
$G(s,A)$ simply by ignoring the extra auxiliary moves made by $\II$ in $G_{\text{unfolded}}(s, F)$.
To handle the case where $\I$ has a winning strategy in $G_{\text{unfolded}}(s, F)$ we need
a few technical lemmas.
We will show that for any winning strategy $\sigma$ for player $\I$ in $G_{\text{unfolded}}(s, F)$,
the set $\setof{x \in 2^\omega \suchthat \forall y, (x, y)~\text{is consistent with}~\sigma}$ is large enough
so that player $\I$ can win $G(s, A)$ by playing a scaled measure supported on a closed subset of it.

\begin{lemma}\label{lemma:scaledmeasureboundedawayfromzero}
Let $M$ be a scaled measure.  Let $\epsilon>0$ with $M(\emptyset)>\epsilon$ and define\\
\[A=\setof{t \in 2^{<\omega} \suchthat M(t) <\epsilon\mu(N_t)~\text{and}~\forall s \subsetneq t,
  ~ M(s) \geq\epsilon\mu(N_s)}.\] If we let 
  \[M' (t) = \begin{cases}M(t) -
    \sum_{\substack{s \in A \\ s \supseteq t}} M(s) &
    \text{if}~ \forall u \subseteq t,~ u \not \in A \\ 0 & \exists u \subseteq t,~ u \in A\end{cases}\]
then $M'$ is a scaled measure and $M'(\emptyset) >
M(\emptyset) - \epsilon$.  In particular, we have that if $T=\setof{t \in 2^{<\omega} \suchthat M'(t)>0}$, then $\mu([T]) \geq M(\emptyset) - \epsilon$ and every $t \in T$ satisfies $M(t) \geq \epsilon \mu(N_t)$.
\end{lemma}
\begin{proof}\
First note that $A$ is an antichain, i.e. no two members are
subsequences of each other, thus the additivity of $M'$ is
immediate from the additivity of $M$ and the definition of $M'$.
Now $M'(\emptyset) = M(\emptyset) - \sum_{t \in A}
M(s)$ and $\sum_{t \in A} M(s) < \sum_{t \in A}
\epsilon \mu(N_t) \leq \epsilon$ again since $A$ is an
antichain.  The facts about the tree $T$ are immediate, since $T$ is the support of $M'$, and we removed all nodes $t$ with $M(t) < \epsilon \mu(N_t)$.
\end{proof}

\begin{lemma}\label{lemma:stable}
Let $\sigma$ be a strategy in $G_{\text{unfolded}}(s, \cdot)$ for
$\I$.  Let $p$ be a position in $G_{\text{unfolded}}(s, \cdot)$
consistent with $\sigma$ with $x$-moves $t=x_0 \dots x_n$ and
$y$-moves $y_0 \dots y_k$ and let $y_{k+1}=i\in \omega$.  Let $M$ be
the scaled measure played by $\sigma$ from $p$ in which no more
$y$-moves after $y_0 \dots y_k$ are played.  Let $\epsilon>0$ and let
$S$ be a subtree of $\setof{u \in 2^{<\omega} \suchthat \forall v
\subseteq u~ M(v)>\epsilon\mu(N_v)}$ with $\mu([S])>0$.  Then for every 
$\beta>0$, there is a tree $T \subseteq S$  such that
\begin{enumerate}
\item $\mu([T]) > (1-\beta)\mu([S])$
\item For every $x \in [T]$ there is an initial segment $u$ of $x$
and  a position $p'$ consistent with
$\sigma$ extending $p$ such that the sequence of $x$-moves of $p'$ is $u$, and the
sequence of $y$-moves of $p'$ is $y_0 \dots y_{k+1}$.
\end{enumerate}
\end{lemma}

\begin{proof}
Let $\beta>0$ be fixed and suppose we have the hypotheses of the
lemma.  In particular, we have the tree $S$ and the sequence $t$ and
the new $y$-move $y_{k+1}$. 

Our plan of attack is to consider playing the move $y_{k+1}$ against $\sigma$ at the earliest possible moment and collecting all the the finite sequences of $x$-moves which $\sigma$ would no longer allow us to reach.  We then consider what would happen if we had, instead, concealed the move $y_{k+1}$ for a while, and had first moved against $\sigma$ to one of these newly forbidden nodes.  At this node we then consider revealing $y_{k+1}$ and again observe which nodes $\sigma$ would forbid.  By repeating this process, we then will argue that $\sigma$ cannot stop us from accessing almost all of the tree $S$, with respect to $\mu$.

Now for each $u \in S$ with $t \subseteq u$,
let $M_u$ be the scaled measure on $N_u$ obtained from $\sigma$
by playing $p$ and then no additional $y$ moves until $y_{k+1}$ is played at $u$, and then no further $y$-moves. 
Let $T_0$ be the support of $M_t$. Let $S_0=T_0 \cap S$.
Let $A_0$ be the nodes $u$ of $S$ such that $u \notin T_0$ but $v \in T_0$
for all proper initial segments $v$ of $u$. 
For each $u \in A_0$, let $T_1(u)$
be the support of $M_u$. Let $T_1= \bigcup_u T_1(u)$. Let
$S_1=S \cap (T_0 \cup T_1)$.
Let $A_1$ be the nodes $u$ of $S_0$  such that $u \notin T_1$, but $v \in T_1$
for all proper initial segments $v$ of $u$. 
Continue this process to define the trees $T_\ell$ and $S_\ell$ for all $\ell$.

We claim that there is some $\ell$ such
that $\mu\left([S_k]\right)>(1-\beta)\mu([S])$.  Suppose not, let
$c=\lim_{\ell \to \infty} \frac{\mu([S_\ell])}{\mu([S])} \leq (1-\beta)$.
Let $k_0$ be large enough so that for any $k>k_0$, we have
$\frac{\mu([S_\ell])}{\mu([S])} > c-\frac{\epsilon \beta}{2}$.  For any $\ell >
\ell_0$ we have $\mu([S] \setminus[S_\ell])\geq\beta\mu([S])$, so by
construction, we have that

\begin{equation*}
\begin{split}
\mu ([T_{\ell+1}]) & = \mu(\bigcup_{u\in A_\ell} [T_{\ell+1}(u)])\\
& > \epsilon  \sum_{u\in A_\ell} \mu(N_u)\\
& \geq \epsilon \mu([S]\setminus[S_\ell])\\
& \geq  \epsilon \beta \mu([S]).
\end{split}
\end{equation*}
since $\mu([T_{\ell+1}(u)]) \geq M_u(u)= M(u) >\epsilon \mu(N_u)$ for any $u \in S$. 
On the other hand, by the choice of $\ell_0$ we have
$\mu([T_{\ell+1}] \cap [S]) \leq \frac{\epsilon \beta}{2} \mu([S])$. From this we have that
$\mu ([T_{\ell+1}]\setminus [S]) > \frac{\epsilon \beta}{2} \mu([S])$.
But the sets $[T_{\ell+1}]\setminus [S]$ are pairwise disjoint, and this contradiction
proves the claim, and so the lemma. To see these sets are pairwise disjoint, suppose
$x \in ([T_{\ell+1}]\setminus [S])) \cap ([T_{m+1}]\setminus [S])$ where $\ell<m$.
So, $x \in [T_{m+1}(u)]$ for some $u \in A_m$. So, $u\in S\setminus S_m=S\setminus (T_0\cup\dots \cup T_m)$.
In particular, $u \notin T_{\ell+1}$. But then $x \notin [T_{k+1}]$, a contradiction.

\end{proof}

\begin{proof}[Proof of Theorem~\ref{theorem:unfolding}]

Suppose $\sigma$ is a winning strategy for $\I$ in $G_{\text{unfolded}}(s, F)$.

Take $M_0$ the scaled measure according to $\sigma$ (a winning
strategy for $\I$ in $G_{\text{unfolded}}(s, F)$) in which $\II$ plays
no $y$ moves. Note that $M_0(\emptyset)>s$. Let $\delta=M_0(\emptyset)-s$.

Let $\{ t_i\}_{i \in \omega}$ enumerate $\omega^{<\omega}\setminus\{ \emptyset\}$ with the property that
any proper initial segment of $t$ precedes $t$ in the enumeration.

Apply Lemma \ref{lemma:scaledmeasureboundedawayfromzero} to $M_0$ with some small $\eta<\frac{\delta}{2}$ to obtain a tree $T_0$ on which $M_0(t)\geq \eta \mu(N_t)$.
In particular, $\mu([T_0]) > s+\frac{\delta}{2}$.
Next apply Lemma~\ref{lemma:stable} with $y_0,\dots, y_k$ the empty sequence and
$(y_{k+1})=t_0$ (note $t_0$ is a sequence of length $1$) to produce $T'_0\subseteq T_0$
such that
\begin{enumerate}
\item \label{wa}
$\mu([T'_0]) > s+\frac{\delta}{2}$.
\item \label{wb}
For every $x \in [T'_0]$ and $t$ either $t=t_0$ or $t=\emptyset$, 
there is an initial segment $u$ of $x$  and a position $p=p(x,t)=p(u,t)$
(that is, $p$ is determined just from the initial segment $u$)
consistent with $\sigma$ where the $x$-moves
consist of $u$ and the $y$-moves consist of $t$. 

\end{enumerate}

For every $u \in T'_0$, the proof of Lemma~\ref{lemma:stable} gives a canonical
position $p=p(u,t_0)$, as in (\ref{wb}) above, in which $t_0(0)$ is played at a
particular round $i(u,t_0,0)< |u|$ of the game. Note that if $u \subseteq v$ are in $T'_0$
and $i(u,t_0,0)$ is defined then so $i(v,t_0,0)$ and $i(v,t_0,0)=i(u,t_0,0)$.

We now need to set up the tree to repeat our use of Lemma~\ref{lemma:stable}.  By applying Lemma~\ref{lemma:scaledmeasureboundedawayfromzero} to the scaled measure $M_u$ for each sequence of $x$-moves 
$u \in T'_0$ such that $i(u,t_0,0)=|u|-1$, we may assume that for all $u \in T'_0$
with corresponding position $p=p(u,t_0)$ we have that $\sigma(p)=(m_{u\concat 0}, m_{u\concat 1})$
has total value $m_u=m_{u\concat 0}+ m_{u\concat 1} > \epsilon_0 \mu([N_u])$
for some fixed $\epsilon_0$ small enough to maintain (\ref{wa}) above.

In general, at step $n$ we produce a tree $T_n$ (using
Lemma~\ref{lemma:stable} on $T'_{n-1}$) and then a $T'_n \subseteq T_n$
(using Lemma~\ref{lemma:scaledmeasureboundedawayfromzero})
with $T_n \subseteq T'_{n-1}$ and such that 
\begin{enumerate}
\item 
$\mu([T'_n]) > s+\frac{\delta}{2}$.
\item 
For every $x \in [T'_n]$ and for every $t_m$, $m \leq n$ (or $t=\emptyset$)
there is an inital segment $u$ of $x$  and a position $p=(x,t_m)=p(u,t_m)$
(that is, $p$ is determined just from the initial segemnt $u$)
consistent with $\sigma$ where the $x$ moves
consist of  $u$ and the $y$ moves consist of $t_m$. Also, if $t_m \subseteq t_k$
then $p(x,t_m)\subseteq p(x,t_k)$. 
\item
For $x$, $t_m$, $u$, $p$ as above, if $\sigma(p)=(m_{u\concat 0}, m_{u\concat 1})$
then $m_u=m_{u\concat 0}+ m_{u\concat 1} > \epsilon_n \mu([N_u])$
for some fixed $\epsilon_n>0$. 

\end{enumerate}

Let $T=\bigcap_n T'_n$, so $T$ is a subtree of $2^{<\omega}$, and
$\mu([T])\geq s+\frac{\delta}{2}$. We will show that 
\[[T] \subseteq \setof{x \in 2^\omega \suchthat \forall y, (x, y)~\text{is consistent with}~\sigma}.\] Let $x \in [T]$, we show that for every
$y \in \omega^\omega$ that there is a run $\vec{p}$ according to $\sigma$
which produces the pair $(x,y)$. Let $t_{g(i)}=y \restriction i$. Let
$\vec p=\bigcup_i p(x,t_{g(i)})$. Since $x \in [T]\subseteq [T_{g(i)}]$ for every $i$,
this is well-defined. Notice that $\vec{p}$ has result $(x, y)$ and is consistent with $\sigma$, thus $(x, y) \not\in F$.  Thus $\forall y\in \omega^\omega\ (x,y)\notin F$ and so $x \notin A$. That is, $[T] \subseteq A^c$. Since $\mu([T])>s$,
$\I$ wins $G(s,[T]^c)$ and so wins $G(s,A)$.

\end{proof}

\begin{corollary}\label{cor:analyticmeas}
Analytic subsets of $2^\omega$ are universally measurable.
\end{corollary}

\begin{proof}
Let $\mu$ be a Borel probability measure on $2^\omega$ and consider the measure game for $\mu$.  Let $A$ be an analytic subset of $2^\omega$, let $F \subseteq 2^\omega
\times \omega^\omega$ be closed such that $A = \setof{x \in 2^\omega
\suchthat \exists y \in \omega^\omega ~ (x, y) \in F}$. Then
$G_{\text{unfolded}}(s, F)$ is determined for all $s$, as it is a $\bP^0_2$ game
for $\II$. But since a
player winning the unfolded game implies that the same player wins
$G(s, A)$, then $G(s, A)$ is determined for all $s$, so the inner and
outer measures of $A$ are equal, thus $A$ is measurable.
\end{proof}

Likewise, we immediately have the following consequence.

\begin{corollary}\label{cor:detmeas}
Assuming $\bP^1_n$-determinacy, all $\bS^1_{n+1}$ sets are universally measurable.
\end{corollary}

We should note that there is an argument for Corollary~\ref{cor:detmeas}
using an unfolded version of a so-called \emph{covering} game, which can be
found in \cite{Moschovakis} (6G.12), where it is attributed to Kechris.
This unfolded covering game argument, however,  does not seem to
directly imply Corollary~\ref{cor:analyticmeas}.

We also have that continuous uniformization is a direct consequence of this unfolding.

\begin{theorem}
Let $R\subseteq  2^\omega \times\omega^\omega$ with $\dom(R)=2^\omega$ and suppose $G_{\text{unfolded}}(0, R)$ is determined.
Then for any $\epsilon >0$, $\exists A \subseteq 2^\omega$, with $A$ closed, 
such that $\mu(A) > 1-\epsilon$ and a continuous function $f\colon A\to \omega^\omega$
such that for all $x \in A$ we have $(x,f(x))\in R$. 
\end{theorem}

\begin{proof}
Since $\II$ trivially has a winning strategy in $G(0, \dom(R))$,
from Theorem~\ref{theorem:unfolding} and the determinacy of $G_{\text{unfolded}}(0,R)$ we have a winning strategy
$\tau$ for $\II$ in $G_{\text{unfolded}}(0,R)$. Let $\tau'$ be the strategy
for $\II$ in $G(0, \cdot)$ obtained from $\tau$ by ignoring the extra $y$ moves.
By Lemma~\ref{lemma:rosendal} there is a tree $T$ with $\mu([T])>1-\epsilon$ such that
all $x \in [T]$ is consistent with $\tau'$. Also, for each $u \in T$ there
is a canonical play $p=p(u)$ consistent with $\tau'$ which produces $u$.
For $x \in [T]$, define $y$ by $(x,y)= \bigcup_k \tau(p(x\restriction k))$. The map from $x \in [T]$ to $y$
is clearly continuous, and $R(x,y)$ holds for all $x \in [T]$ as $\tau$ is winning, and the theorem follows.

\end{proof}

\bibliographystyle{amsplain}

\bibliography{measure}

\providecommand{\bysame}{\leavevmode\hbox to3em{\hrulefill}\thinspace}
\providecommand{\MR}{\relax\ifhmode\unskip\space\fi MR }
\providecommand{\MRhref}[2]{%
  \href{http://www.ams.org/mathscinet-getitem?mr=#1}{#2}
}
\providecommand{\href}[2]{#2}
\begin{thebibliography}{10}

\bibitem{erdos-renyi}
P.~Erd\"{o}s and A.~R\'{e}nyi, \emph{On {C}antor's series with convergent
  $\sum\frac{1}{q_n}$.}, Ann. Univ. Sci. Budapest. E\"{o}tv\"{o}s Sect. Math.
  \textbf{2} (1959), 93--109.

\bibitem{harmon}
Glyn Harmon, \emph{Metric {N}unber {T}heory}, London {M}athematical {S}ociety
  {M}onographs, Clarendon {P}ress {P}ublications, 1998.

\bibitem{kochen-stone}
Simon Kochen and Charles Stone, \emph{A {Note} on the {B}orel-{C}antelli
  lemma}, Illinois Journal of Math. \textbf{8} (1964), 248--251.

\bibitem{lamperti}
J.~Lamperti, \emph{Wiener's test and {M}arkov chains}, Journal of Mathematical
  Analysis and applications \textbf{6} (1963), 58--66.

\bibitem{Martin_determinacy}
Donald~A. Martin, \emph{A purely inductive proof of {B}orel determinacy},
  Recursion theory (Ithaca, N.Y., 1982), Proc. Sympos. Pure Math., vol.~42,
  Amer. Math. Soc., Providence, RI, 1985, pp.~303--308.

\bibitem{Martin1998}
\bysame, \emph{The detereminacy of {B}lackwell games}, The Journal of Symbolic
  Logic \textbf{63} (1998), no.~4, 1565--1581.

\bibitem{Moschovakis}
Yiannis~N. Moschovakis, \emph{Descriptive set theory}, Mathematical Surveys and
  Monographs, American Mathematical Society; 2nd edition, 2009.

\bibitem{Oxtoby}
J.~C. Oxtoby, \emph{Measure and category: {A} survey of the analogies between
  topological and measure spaces}, Graduate Texts in Mathematics, vol.~2,
  Springer-Verlag, New York-Berlin, 1971.

\bibitem{petrov}
Valentin~V. Petrov, \emph{A generalization of the {B}orel-{C}antelli {Lemma}},
  Statistics \& Probability Letters \textbf{67} (2004), 233--239.

\bibitem{rosendal}
Christian Rosendal, \emph{Games and {L}ebesgue {M}easurability}, 2009,
  \url{http://homepages.math.uic.edu/~rosendal/WebpagesMathCourses/MATH511-notes/DST%20notes%20-%20Lebesgue%20measurability%2002.pdf}.

\bibitem{spitzer}
F.~Spitzer, \emph{Principles of {R}andom {W}alk}, Van Nostrand, Princeton,
  1964.

\end{thebibliography}

\end{document}